\newif\ifarxivversion
\arxivversiontrue

\ifarxivversion
\documentclass{amsart}
\else
\documentclass[review,onefignum,onetabnum]{siamart190516}
\fi

\usepackage[letterpaper,hmargin=0.9in,vmargin=0.8in]{geometry}
\usepackage{amsmath}
\usepackage{amssymb}
\usepackage{mathtools}
\usepackage{mathbbol}
\usepackage{soul}
\usepackage{tikz}
\usetikzlibrary{matrix,arrows,cd,calc}
\usepackage{mathbbol}
\usepackage{thm-restate}

\usepackage{wrapfig}
\usepackage{graphicx}
\graphicspath{{figures/}}
\usepackage{color}
\usepackage[all]{xy}
\usepackage{enumitem}

\usepackage{xcolor}
\definecolor{darkgreen}{rgb}{0,0.45,0}
\ifarxivversion
\usepackage{hyperref}
\fi
\hypersetup{colorlinks,urlcolor=blue,citecolor=darkgreen,linkcolor=darkgreen,linktocpage}
\usepackage[capitalize]{cleveref}

\ifarxivversion
\newtheorem{theorem}{Theorem}[section]
\newtheorem{lemma}[theorem]{Lemma}

\newtheorem{proposition}[theorem]{Proposition}
\newtheorem{corollary}[theorem]{Corollary}
\newtheorem{question}[theorem]{Question}

\theoremstyle{definition}
\newtheorem{definition}[theorem]{Definition}
\newtheorem{example}[theorem]{Example}

\newtheorem{remark}[theorem]{Remark}

\else
\newsiamthm{remark}{Remark}
\newsiamthm{question}{Question}

\newsiamremark{example}{Example}
\newsiamremark{construction}{Construction}
\fi

\ifarxivversion
\else
\headers{On the stability of multigraded Betti numbers and Hilbert functions}{S. Oudot and L. Scoccola}
\fi

\ifarxivversion
\title[On the stability of multigraded Betti numbers and Hilbert functions]{On the stability of multigraded Betti numbers\\and Hilbert functions}
\else
\title{On the stability of multigraded Betti numbers and Hilbert functions\thanks{Submitted to the editors DATE.
\funding{L.S.~was partially supported by the National Science Foundation through grants CCF-2006661
and CAREER award DMS-1943758.}}}
\fi

\ifarxivversion
\author{Steve Oudot}
\address{Inria Saclay, Palaiseau, France}
\author{Luis Scoccola}
\address{Mathematical Institute, University of Oxford, United Kingdom}
\else
\author{Steve Oudot\thanks{Inria Saclay, Palaiseau, France.}
\and Luis Scoccola\thanks{Mathematical Institute, University of Oxford, United Kingdom.}}
\fi

\newcommand{\define}[1]{\emph{#1}}

\newcommand{\Ksf}{\mathsf{K}}

\newcommand{\gldim}{\mathsf{gl.dim}}

\newcommand{\betti}{{\beta}}
\newcommand{\sbetti}{{\betti}\mathcal{B}}
\newcommand{\msd}{{\mathcal{HB}}}
\newcommand{\dmax}{\widehat{d}_\mathsf{m}}
\newcommand{\cost}{\mathsf{cost}}

\newcommand{\dwass}{d_{W^1}}
\newcommand{\dmatch}{d_{\mathsf{match}}}
\newcommand{\dwassp}[1]{d_{W^{#1}}}
\newcommand{\dwasshils}{d^\hils_{W^1}}

\newcommand{\dIone}{{d_I^1}}
\newcommand{\dIonehat}{\overline{\dIone}}

\newcommand{\hil}{\mathsf{Hil}}
\newcommand{\hils}{\mathbf{Hils}}
\newcommand{\even}{{2\N}}
\newcommand{\odd}{{2\N+1}}

\newcommand{\cells}{\mathsf{Cells}}

\renewcommand{\leq}{\leqslant}
\renewcommand{\geq}{\geqslant}
\renewcommand{\ker}{\mathsf{ker}}
\newcommand{\coker}{\mathsf{coker}}
\newcommand{\im}{\mathsf{im}}

\newcommand{\vect}{\mathbf{vec}}
\newcommand{\ab}{\mathbf{Ab}}
\newcommand{\persmod}{\mathbf{pmod}}

\newcommand{\sbarcodes}{\mathbf{sBarc}}

\newcommand{\matt}{\mu}
\newcommand{\field}{\mathbb{k}}

\newcommand{\R}{\mathbb{R}}
\newcommand{\Z}{\mathbb{Z}}
\newcommand{\N}{\mathbb{N}}

\newcommand{\RR}{\mathbf{R}}

\newcommand{\BBB}{\mathcal{B}}
\newcommand{\B}{\BBB}
\newcommand{\A}{\mathcal{A}}
\newcommand{\CCC}{\mathcal{C}}
\newcommand{\C}{\CCC}
\newcommand{\DDD}{\mathcal{D}}
\newcommand{\D}{\DDD}
\newcommand{\EEE}{\mathcal{E}}

\newcommand{\III}{\mathcal{I}}
\newcommand{\JJJ}{\mathcal{J}}
\newcommand{\LLL}{\mathcal{L}}

\newcommand{\PPP}{\mathcal{P}}
\newcommand{\PP}{\mathsf{P}}
\newcommand{\QQ}{\mathsf{Q}}

\newcommand{\XXX}{\mathcal{X}}

\renewcommand{\epsilon}{\varepsilon}
\renewcommand{\phi}{\varphi}

\DeclareMathOperator*{\Hom}{Hom}

\def\noteson{\gdef\luis##1{\noindent{\color{blue}[Luis: ##1]}}\gdef\steve##1{\noindent{\color{red}[Steve: ##1]}}}

\noteson

\ifarxivversion
\else
\ifpdf
\hypersetup{
  pdftitle={On the stability of multigraded Betti numbers and Hilbert functions},
  pdfauthor={S. Oudot and L. Scoccola}
}
\fi
\fi

\begin{document}

\maketitle

\begin{abstract}
    Multigraded Betti numbers are one of the simplest invariants of multiparameter persistence modules.
    This invariant is useful in theory---it completely determines the Hilbert function of the module and the isomorphism type of the free modules in its minimal free resolution---as well as in practice---it is easy to visualize and it is one of the main outputs of current multiparameter persistent homology software, such as RIVET.
    However, to the best of our knowledge, no stability result with respect to the interleaving distance has been established for this invariant so far, and this potential lack of stability limits its practical applications.
    We prove a stability result for multigraded Betti numbers, using an efficiently computable bottleneck-type dissimilarity function we introduce.
    Our notion of matching is inspired by recent work on signed barcodes, and allows matching bars of the same module in homological degrees of different parity, in addition to matchings bars of different modules in homological degrees of the same parity.
    Our stability result is a combination of Hilbert's syzygy theorem, Bjerkevik's bottleneck stability for free modules, and a novel stability result for projective resolutions.
    We also prove, in the two-parameter case, a $1$-Wasserstein stability result for Hilbert functions with respect to the $1$-presentation distance of Bjerkevik and Lesnick.
\end{abstract}


\section{Introduction}

\subsection*{Context}

The study of invariants coming from resolutions of persistence modules is among the promising directions for the development and application of multiparameter persistence~\cite{harrington-otter-schenck-tillmann,lesnick-wright,miller2020homological}. The case of free resolutions is particularly appealing in practice, since free modules can be completely described by a set of generators, so they are easy to encode and manipulate on a computer. And, while a persistence module may admit infinitely many distinct free resolutions, under suitable (and mild) conditions, there exists a \emph{minimal free resolution}, which is unique up to isomorphism and therefore has a unique associated set of generators---the so-called \emph{multigraded Betti numbers}. Multigraded Betti numbers are a well-known homological invariant of graded modules, and are used in the context of topological data analysis as a tool for the visualization and exploration of the structure of multiparameter persistence modules~\cite{lesnick-wright, rivet}.

\begin{figure}[htb]
    \centering

    {\def\svgwidth{.5\textwidth}
    {
\begingroup%
  \makeatletter%
  \providecommand\color[2][]{%
    \errmessage{(Inkscape) Color is used for the text in Inkscape, but the package 'color.sty' is not loaded}%
    \renewcommand\color[2][]{}%
  }%
  \providecommand\transparent[1]{%
    \errmessage{(Inkscape) Transparency is used (non-zero) for the text in Inkscape, but the package 'transparent.sty' is not loaded}%
    \renewcommand\transparent[1]{}%
  }%
  \providecommand\rotatebox[2]{#2}%
  \newcommand*\fsize{\dimexpr\f@size pt\relax}%
  \newcommand*\lineheight[1]{\fontsize{\fsize}{#1\fsize}\selectfont}%
  \ifx\svgwidth\undefined%
    \setlength{\unitlength}{392.60358018bp}%
    \ifx\svgscale\undefined%
      \relax%
    \else%
      \setlength{\unitlength}{\unitlength * \real{\svgscale}}%
    \fi%
  \else%
    \setlength{\unitlength}{\svgwidth}%
  \fi%
  \global\let\svgwidth\undefined%
  \global\let\svgscale\undefined%
  \makeatother%
  \begin{picture}(1,0.36914973)%
    \lineheight{1}%
    \setlength\tabcolsep{0pt}%
    \put(0,0){\includegraphics[width=\unitlength,page=1]{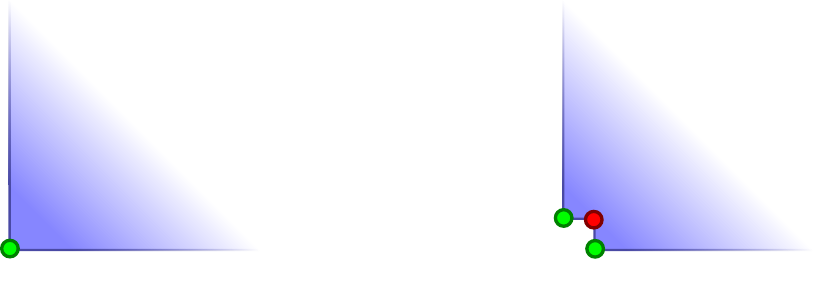}}%
    \put(0.16266706,0.27492381){\makebox(0,0)[lt]{\lineheight{1.25}\smash{\begin{tabular}[t]{l}$M$\end{tabular}}}}%
    \put(0.00162814,0.00512403){\makebox(0,0)[lt]{\lineheight{1.25}\smash{\begin{tabular}[t]{l}$(0,0)$\end{tabular}}}}%
    \put(0.84328572,0.27492376){\makebox(0,0)[lt]{\lineheight{1.25}\smash{\begin{tabular}[t]{l}$N$\end{tabular}}}}%
    \put(0.71999974,0.0059014){\makebox(0,0)[lt]{\lineheight{1.25}\smash{\begin{tabular}[t]{l}$(\epsilon,0)$\end{tabular}}}}%
    \put(0.56654191,0.12817789){\makebox(0,0)[lt]{\lineheight{1.25}\smash{\begin{tabular}[t]{l}$(0,\epsilon)$\end{tabular}}}}%
    \put(0.7338487,0.12817789){\makebox(0,0)[lt]{\lineheight{1.25}\smash{\begin{tabular}[t]{l}$(\epsilon,\epsilon)$\end{tabular}}}}%
  \end{picture}%
\endgroup%
}}

    \caption{The multigraded Betti numbers of $M$ and $N$ are $\betti_0(M) = \{(0,0)\}$ (green) and $\betti_k(M) = \emptyset$ for $k \geq 1$, and $\betti_0(N) = \{(\epsilon,0),(0,\epsilon)\}$ (green), $\betti_1(N) = \{(\epsilon,\epsilon)\}$ (red), and $\betti_k(N) = \emptyset$ for $k \geq 2$. Although $d_I(M,N) \leq \epsilon$, there is no complete matching between the Betti numbers of $M$ and $N$.
    As a result, the bottleneck distance between, e.g., $\betti_0(M)$ and $\betti_0(N)$ is infinite, as any unmatched free summand is infinitely persistent.
    }
    \label{fig:Betti-matching}
\end{figure}
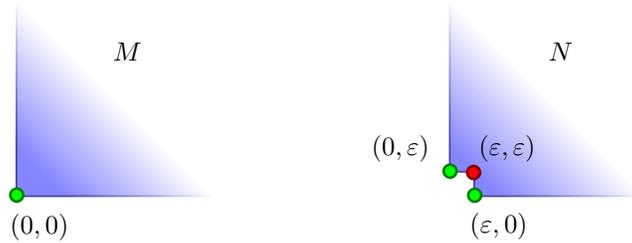

One current limitation for the use of multigraded Betti numbers in applications is that they do not seem to satisfy a bottleneck stability result analogous to the bottleneck stability of persistence barcodes for one-parameter persistence modules. Indeed, simple examples such as the one from \cref{fig:Betti-matching} reveal that modules that are arbitrarily close in the interleaving distance can have multigraded Betti numbers that are infinitely far apart in terms of matching cost.
More generally, while free resolutions are known to be stable in an interleaving distance defined on the homotopy category~\cite{berkouk2019stable}, they currently lack a corresponding bottleneck stability property. Stability is important in applications, where it allows one to bound the dissimilarity between persistence modules from below by the dissimilarity between their invariants. In the case of bottleneck stability, the dissimilarity between the invariants has a simple combinatorial formulation and is therefore easy to compute. Also importantly, in the context of one-parameter persistence, bottleneck-type distances turn the space of persistence barcodes into a space of discrete measures, equipped with an optimal transport distance, thus enabling the development of several mathematical frameworks for doing statistics, differential calculus, optimization, and machine learning with persistence barcodes---see~\cite{chazal2021introduction} for a survey.

Our approach to the problem of establishing a bottleneck stability result for multigraded Betti numbers is inspired by recent work on signed barcodes and their stability~\cite{botnan-oppermann-oudot}, which takes its roots in the line of work on generalized persistence diagrams~\cite{asashiba2019approximation,kim2018generalized,mccleary2021edit,patel2018generalized}. Going back to the example from \cref{fig:Betti-matching}, we see that there is a matching of cost $\epsilon$ if we
allow for the matching of generators coming from the same module but in different degrees, e.g., matching $(0,0) \in \betti_0(M)$ with $(0,\epsilon) \in \betti_0(N)$ and $(\epsilon, \epsilon)\in\betti_1(N)$  with $(\epsilon, 0)\in\betti_0(N)$.
Specifically, we allow for the matching of generators in even (homological) degrees with generators in odd degrees within the same resolution. This gives rise to a notion of \emph{signed barcode} coming from a free resolution, and to a corresponding \emph{bottleneck dissimilarity function}.
We address the aforementioned limitation  by proving a stability result for multigraded Betti numbers using the bottleneck dissimilarity. In the one- and two-parameter setting, we also prove a stability result using the \emph{signed $1$-Wasserstein distance}, a $1$-Wasserstein version of the bottleneck dissimilarity.

\subsection*{Mathematical framework}

For relevant concepts not defined in this section, we refer the reader to \cref{background-section}.
The stability results presented in this paper can be framed using the notion of \define{decomposition}, introduced in \cite{botnan-oppermann-oudot} in the case of the rank invariant.
Decompositions allow us to represent algebraic invariants of multiparameter persistence modules---in the case of this paper the Hilbert function---by means of a geometric descriptor: a signed barcode, essentially consisting of a collection of signed points in Euclidean space.
We now introduce these notions.

Let $\field$ be a fixed field, and denote by $\vect$ the category of finite-dimensional vector spaces over~$\field$. The \define{Hilbert function} of a multiparameter persistence module $M : \RR^n \to \vect$, denoted by $\hil(M) : \RR^n \to \Z$, is defined by $\hil(M)(i) = \dim(M(i))$.
A finite \define{Hilbert decomposition} of a function $\eta : \RR^n \to \Z$ consists of a pair $(P,Q)$ of free $n$-parameter persistence modules of finite rank such that $\eta = \hil(P) - \hil(Q)$.

Recall that, for every free $n$-parameter persistence module $P$, there exists a unique multiset of elements of $\RR^n$,
called the \define{barcode} of $P$ and denoted $\B(P)$, that satisfies $P \cong \bigoplus_{i \in \B(P)} F_i$,
where $F_i$ denotes the interval module with support $\{j \in \RR^n : j \geq i\}$.
In particular, any Hilbert decomposition $(P,Q)$ gives rise to a pair of barcodes $(\B(P),\B(Q))$.

We consider two ways of constructing a Hilbert decomposition of $\hil(M)$, given any finitely presentable multiparameter persistence module $M$.
Both constructions give Hilbert decompositions that are unique up to isomorphism, and that are, in their own sense, \emph{minimal}:

\smallskip
\begin{enumerate}[leftmargin=5.5mm]
    \item[1.] \label{construction-1} Choose a Hilbert decomposition $(P^*,Q^*)$ of $\hil(M)$ that is minimal in the sense that $\B(P^*)$ and $\B(Q^*)$ are disjoint (as multisets).
    \medskip

    \item[2.] \label{construction-2} Note that any finite free resolution $P_\bullet \to M$ gives rise to a Hilbert decomposition $\left(\bigoplus_{k \in \even} P_{k},\bigoplus_{k \in \odd} P_k\right)$, by the rank-nullity theorem from linear algebra,
    and choose $P_\bullet \to M$ to be a minimal free resolution.
\end{enumerate}
\smallskip

We interpret pairs of barcodes $(\B,\C)$ as signed barcodes and think of $\B$ as the positive part and of $\C$ as the negative part.
We refer to these signed barcodes as \define{$n$-dimensional signed barcodes}.

Construction $(1.)$ gives, for any fixed multiparameter persistence module $M$, a signed barcode $\msd(M) = (\B(P^*),\B(Q^*))$, which we call the \define{minimal Hilbert decomposition signed barcode} of $M$.

Construction $(2.)$ gives a signed barcode $\sbetti(M) = \left(\B\left(\bigoplus_{k \in \even} P_{k}\right),\B\left(\bigoplus_{k \in \odd} P_k\right)\right)$, which we call  the \define{Betti signed barcode} of $M$.
The name comes from the observation that $\sbetti(M) =  (\betti_{\even}(M),\betti_{\odd}(M))$, where $\betti_{\even}(M)$ and $\betti_{\odd}(M)$ denote the multigraded Betti numbers of $M$ in even homological degrees and in odd homological degrees, respectively.

We study the stability of these two constructions, in the case of finitely presentable modules and thus of finite barcodes. 

\subsection*{Contributions}
Let $\B$ and $\C$ be finite multisets of elements of $\RR^n$ and let $\epsilon \geq 0$.
An \define{$\epsilon$-bijection} between $\B$ and $\C$ is a bijection $h : \B \to \C$ of multisets of elements of $\RR^n$, with the property that, for every $i \in \B$, we have $\|i - h(i)\|_\infty \leq \epsilon$.%
\footnote{See \cref{background-section} for details about multisets.
Note that our notion of $\epsilon$-bijection does not allow for ``unmatched bars''---as is common in the persistence literature---since bars corresponding to free modules have infinite persistence.}
We define the \define{bottleneck distance} on barcodes by
\[
    d_B(\B,\C) = \inf\big\{\,\epsilon \geq 0 \,: \,\text{there exists an $\epsilon$-bijection }\, \B \to \C\,\big\} \in \R_{\geq 0}\cup\{\infty\},
    \]
    which is known to be an extended pseudodistance. From it we can derive the \define{bottleneck dissimilarity function}~$\widehat{d_B}$ on signed barcodes: for any finite signed barcodes $\B=(\B_+, \B_-)$ and $\C=(\C_+, \C_-)$, we let
    \begin{equation}
        \label{definition-dbhat}
    \widehat{d_B}(\B,\C) = d_B(\B_+ \cup \C_-,\ \C_+ \cup \B_-)\in\R_{\geq 0}\cup\{\infty\}.
    \end{equation}
    
A comment about the definition of $\widehat{d_B}$ is in order.
Readers familiar with the optimal transport literature may notice that the process of
extending the distance $d_B$ on unsigned barcodes to the dissimilarity $\widehat{d_B}$ on signed
barcodes is analogous to that of extending an optimal transport distance on positive measures to a dissimilarity on signed measures, such as the Kantorovich norm \cite{kantorovich-rubinstein}; see \cite{ambrosio-mainini-serfaty} for an example in the context of optimal transport, and \cite{botnan-oppermann-oudot,bubenik-elchesen,divol-lacombe} and \cref{proposition:wasserstein-is-kantorovich} for examples in the persistence literature.
For our purposes, it is useful to abstract the following notion.
We say that a dissimilarity function~$\widehat{d}$ on signed barcodes is \define{balanced} if
\begin{equation}\label{balanced-dissimilarity}
    \widehat{d}\big((\B_+, \B_- \cup \D), (\C_+, \C_-)\big) = \widehat{d}\big((\B_+, \B_-), (\C_+ \cup \D, \C_-)\big)
\end{equation}
for all signed barcodes
$(\B_+, \B_-)$, $(\C_+, \C_-)$ and unsigned barcode $\D$.
As can be easily checked, this is equivalent to saying that $\widehat{d}$ is induced by some dissimilarity function $d$ on unsigned barcodes, in the sense that 
\[
\widehat{d}\big((\B_+, \B_-),(\C_+, \C_-)\big) = d(\B_+ \cup \C_-, \C_+\cup \B_-)
\]
for all signed barcodes $(\B_+, \B_-)$ and $(\C_+, \C_-)$.  In particular, $\widehat{d_B}$ is balanced.
A main motivation for considering balanced distances, beside the stability results proven here, is that their computation relies on the computation of known distances between (unsigned) barcodes---see \cref{algorithm-section}.

In \cref{stability-projective-resolutions-section}, we prove the following stability theorem for Betti signed barcodes in the bottleneck dissimilarity, which can be interpreted as a bottleneck stability result for multigraded Betti numbers.
In the result, and in the rest of this introduction, $d_I$ stands for the interleaving distance between multiparameter persistence modules~\cite{lesnick}, recalled in \cref{background-section}.
\begin{restatable}{theorem}{maintheorem}
    \label{main-theorem}
    Let $n \geq 2$.
    For finitely presentable modules $M,N : \RR^n \to \vect$, we have
    \[
        \widehat{d_B}\big(\sbetti(M),\, \sbetti(N)\big)\, \leq\, (n^2 - 1)\cdot d_I(M,N).
    \]
    In other words,
    \[
        d_B\big(\betti_{\even}(M) \cup \betti_{\odd}(N)\,,\, \betti_{\even}(N) \cup \betti_{\odd}(M)\big)\, \leq\, (n^2 - 1)\cdot d_I(M,N).
    \]
\end{restatable}

Note that, when $n=1$, the usual isometry theorem for persistence barcodes (see, e.g., \cite[Theorem~3.4]{lesnick}) implies that $\widehat{d_B}\big(\sbetti(M),\, \sbetti(N)\big)\, \leq\, 2\cdot d_I(M,N)$; see \cref{stability-n=1}.

\medskip

\noindent Our proof of \cref{main-theorem} combines three key results:
\begin{enumerate}
  \item Hilbert's global dimension bound $\gldim(\field[x_1,\dots,x_n]) \leq n$, recalled as \cref{hilberts-theorem}.
  \item Bjerkevik's bottleneck stability for free modules, recalled as \cref{bjerkevik-free-modules}.
  \item An interleaving stability result for projective resolutions, of independent interest, which we prove in this paper and state as \cref{stability-resolutions} below.
\end{enumerate}
The third result extends Schanuel's lemma for projective resolutions---a well-known result in homological algebra---from modules to persistence modules, and recovers the original result when $\epsilon = 0$.

\begin{restatable}{proposition}{stabilityresolutions}
    \label{stability-resolutions}
Given $n\in\N$, let $M,N : \RR^n \to \vect$ be persistence modules, and let $P_\bullet \to M$ and $Q_\bullet \to N$ be projective resolutions of finite length.
    If $M$ and $N$ are $\epsilon$-interleaved, then so are
    \[
        P_0 \oplus Q_1[\epsilon] \oplus P_2[2\epsilon] \oplus Q_3[3\epsilon] \oplus \cdots
        \;\;\;\; \text{and} \;\;\;\;
        Q_0 \oplus P_1[\epsilon] \oplus Q_2[2\epsilon] \oplus P_3[3\epsilon] \oplus \cdots
    \]
\end{restatable}
As usual, the $\epsilon$-shift of a persistence module $M \colon \RR^n \to \vect$ is defined pointwise by $M[\epsilon](r) = M(r+\epsilon)$.
We remark that, by a recent result of Geist and Miller, all multiparameter persistence modules admit projective resolutions of finite length \cite{geist-miller}.

\medskip

It is interesting to note how \cref{main-theorem} addresses the difficulty in the example of \cref{fig:Betti-matching}.
\cref{main-theorem} also addresses a problem pointed out in \cite[Example~9.1]{botnan-lesnick}, which shows that there need not be a low cost matching between the indecomposable summands of two modules at small interleaving distance (see \cref{fig:Betti-matching-intervals}).

\begin{figure}[tb]
    \centering

    {\def\svgwidth{.8\textwidth}
    {
\begingroup%
  \makeatletter%
  \providecommand\color[2][]{%
    \errmessage{(Inkscape) Color is used for the text in Inkscape, but the package 'color.sty' is not loaded}%
    \renewcommand\color[2][]{}%
  }%
  \providecommand\transparent[1]{%
    \errmessage{(Inkscape) Transparency is used (non-zero) for the text in Inkscape, but the package 'transparent.sty' is not loaded}%
    \renewcommand\transparent[1]{}%
  }%
  \providecommand\rotatebox[2]{#2}%
  \newcommand*\fsize{\dimexpr\f@size pt\relax}%
  \newcommand*\lineheight[1]{\fontsize{\fsize}{#1\fsize}\selectfont}%
  \ifx\svgwidth\undefined%
    \setlength{\unitlength}{602.06969762bp}%
    \ifx\svgscale\undefined%
      \relax%
    \else%
      \setlength{\unitlength}{\unitlength * \real{\svgscale}}%
    \fi%
  \else%
    \setlength{\unitlength}{\svgwidth}%
  \fi%
  \global\let\svgwidth\undefined%
  \global\let\svgscale\undefined%
  \makeatother%
  \begin{picture}(1,0.3027246)%
    \lineheight{1}%
    \setlength\tabcolsep{0pt}%
    \put(0,0){\includegraphics[width=\unitlength,page=1]{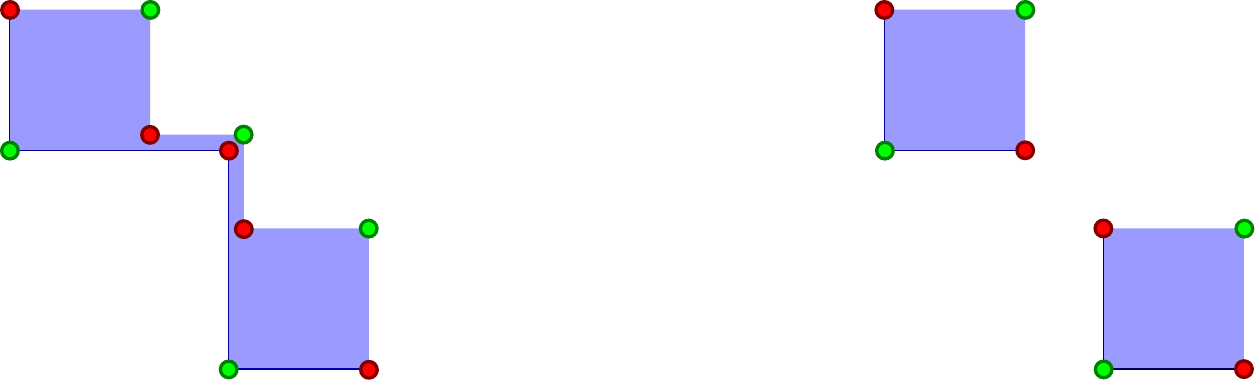}}%
    \put(0.10523597,0.1296393){\color[rgb]{0,0,0}\makebox(0,0)[lt]{\lineheight{1.25}\smash{\begin{tabular}[t]{l}$M$\end{tabular}}}}%
    \put(0.8213092,0.13111003){\color[rgb]{0,0,0}\makebox(0,0)[lt]{\lineheight{1.25}\smash{\begin{tabular}[t]{l}$N$\end{tabular}}}}%
  \end{picture}%
\endgroup%
}}

    \caption{
    In \cite[Example~9.1]{botnan-lesnick}, it is shown that a straightforward extension of the bottleneck distance to multiparameter interval decomposable modules (see \cref{background-section}) is not stable with respect to the interleaving distance.
    More specifically, it is shown that one can construct interval decomposable modules $M$ and $N$ such that $d_I(M,N)$ is arbitrarily small, and such that $M$ is indecomposable, $N$ decomposes into a direct sum of two indecomposable modules $N_1$ and $N_2$, and $d_I(M,N_1)$, $d_I(M,N_2)$, $d_I(N_1,0)$, and $d_I(N_2,0)$ are all large.
    Here, we illustrate a similar example, including the Betti signed barcodes $M$ and $N$.
    Even Betti numbers are shown in green and odd Betti numbers in red.
    As \cref{main-theorem} guarantees, there is a low-cost matching between the Betti numbers of $M$ and $N$, but, in order to construct this matching, the two Betti numbers of $M$ located at the corners of the thin region of its support must be matched together.}
    \label{fig:Betti-matching-intervals}
\end{figure}

\medskip

Note that $\widehat{d_B}$, as defined in \cref{definition-dbhat}, is only a dissimilarity function, as it does not satisfy the triangle inequality, even when restricted to Betti signed barcodes of persistence modules---see \cref{possible-issue-example} and \cref{triangle-inequality-proof-for-bottleneck}.
Nevertheless, in \cref{weak-universality-section}, we prove that $\widehat{d_B}$ is in fact universal, up to bi-Lipschitz equivalence, among all the balanced dissimilarity functions on signed barcodes which are stable, that is, which provide a lower bound for the interleaving distance in the following sense.
Recall that a \define{dissimilarity function} on a set $A$ is any symmetric function $A \times A \to [0,\infty]$ whose value is zero on every pair $(a,a)$ where $a\in A$. 
A dissimilarity function~$\widehat{d}$ on (finite)
signed barcodes is \define{stable} if for all finitely presentable $M$ and $N$ we have
\[
    \widehat{d}(\sbetti(M),\sbetti(N)) \leq d_I(M,N).
\]
Note that $\widehat{d_B}/(n^2-1)$ is stable when $n\geq 2$ by \cref{main-theorem}, and that $\widehat{d_B}/2$ is stable when $n=1$.

\begin{restatable}{proposition}{weakuniversality}
\label{weak-universality}
    The collection of balanced and stable dissimilarity functions on $n$-dimensional signed barcodes has a maximum with respect to the pointwise order; denote it by $\dmax$.
    We have a Lipschitz equivalence: $\widehat{d_B}/2 \,\,\leq \,\, \dmax \leq \widehat{d_B}$ when $n=1$, and $\widehat{d_B}/(n^2-1) \,\,\leq \,\, \dmax \leq \widehat{d_B}$ when $n\geq 2$.
\end{restatable}

We also use the universality of the bottleneck dissimilarity to prove the following no-go result, which says that a dissimilarity on Betti signed barcodes that is stable, balanced, and that satisfies the triangle inequality gives a trivial lower bound for the interleaving distance.
Thus, if one wants a non-trivial lower bound for the interleaving distance that only takes Betti signed barcodes into consideration, and that is balanced, then a dissimilarity is the best that one can get.

\begin{restatable}{proposition}{nogothm}
    \label{no-go-thm}
    Let $\widehat{d}$ be a dissimilarity on $n$-dimensional signed barcodes.
    Assume that $\widehat{d}$ is stable and balanced, and that it satisfies the triangle inequality.
    Then, for any finitely presentable $M,N : \RR^n \to \vect$ with $d_I(M,N) < \infty$, we have $\widehat{d}(\sbetti(M), \sbetti(N)) = 0$.
\end{restatable}

We remark that \cref{no-go-thm} also applies to the case $n=1$.
Of course, in the case $n=1$, the usual isometry theorem for persistence barcodes gives a non-trivial lower bound for the interleaving distance: the usual bottleneck distance for one-parameter persistence barcodes.
This does not contradict \cref{no-go-thm}, since the usual bottleneck distance for one-parameter persistence barcodes is not defined at the level of Betti numbers.

\medskip

In \cref{instability-section}, we turn our focus to signed barcodes coming from minimal Hilbert decompositions.
We show that, while minimal Hilbert decompositions exist (\cref{existence-minimal-hilbert}), they are not stable in the bottleneck dissimilarity, in the sense that there is no constant $c \geq 1$ such that $\widehat{d_B}(\msd(M), \msd(N)) \leq \, c\cdot d_I(M,N)$ for all finitely presentable modules $M$ and $N$ (\cref{instability-minimal-decomposition-example}). The intuition behind this result is that cancelling out all the bars that are common to the positive and negative parts of the signed barcode makes it sometimes impossible to build matchings of low bottleneck cost between signed barcodes coming from nearby persistence modules.

Our last main result shows that this limitation can be lifted, at least in the one- and two-parameter setting, by replacing the bottleneck dissimilarity with the \emph{signed $1$-Wasserstein distance}~$\widehat{\dwass}$ (\cref{signed-wass-def})
induced by the $1$-Wasserstein distance~$\dwass$ on barcodes:


\begin{restatable}{theorem}{wassersteinstability}
    \label{wasserstein-stability}
    Let $n \in \{1,2\}$.
    For finitely presentable $n$-parameter persistence modules $M$ and $N$ we have
    \[
        \widehat{\dwass}(\msd(M),\msd(N)) = \widehat{\dwass}(\sbetti(M),\sbetti(N)) \leq n \cdot \dIone(M,N).
    \]
\end{restatable}

In the result, $\dIone$ stands for the \emph{$1$-presentation distance} (\cref{presentation-distance-def}), which is the $p=1$ case of the $p$-presentation distance recently introduced by Bjerkevik and Lesnick~\cite{bjerkevik-lesnick}.
Note that the signed $1$-Wasserstein distance does satisfy the triangle inequality (\cref{properties-signed-wass})---justifying its name---and can be used to compare multiparameter persistence modules directly at the level of Hilbert functions.
Indeed,
we explain in \cref{hilbert-stability-section} 
how \cref{wasserstein-stability} can be interpreted as a stability result for Hilbert functions.
The proof of \cref{wasserstein-stability}, given in~\cref{stability-hilbert-decompositions}, relies on key properties of the $p$-presentation distance established by Bjerkevik and Lesnick in~\cite{bjerkevik-lesnick}.
Note that the equality in \cref{wasserstein-stability} says that $\widehat{\dwass}$ does not distinguish between our two types of minimal decompositions of the Hilbert function.
This equality is straightforward to prove, and the non-trivial part of \cref{wasserstein-stability} is the inequality.

We also show in \cref{proposition:wasserstein-is-kantorovich} that the signed $1$-Wasserstein distance can be seen as a particular case of the distance on signed measures induced by the Kantorovich norm.

\medskip

In \cref{algorithm-section}, we address the efficient computability of the lower bounds for the interleaving and $1$-presentation distances provided by \cref{main-theorem} and \cref{wasserstein-stability}, respectively.
We also show in \cref{sec:compute_minHil} how M\"obius inversion can be used to efficiently compute minimal Hilbert decompositions directly, without resorting to the computation of multigraded Betti numbers; see \cref{remark:computation-convolution}.

\medskip

In \cref{consequences-section}, we derive some extensions of our results, including a stability result for Hilbert functions (\cref{hilbert-stability-corollary}), stability results for multigraded Betti numbers of sublevel set persistence (\cref{stability-betti-sublevel}), and a generalization of \cref{main-theorem} to signed barcodes coming from other notions of minimal resolution (\cref{relative-stability-theorem}), which uses the language of relative homological algebra.

\medskip

We conclude the paper with a discussion in \cref{discussion-section} about some of the implications and perspectives of our work.

\subsection*{Related work}

Our use of free resolutions to construct invariants of multiparameter persistence modules is tightly related to the  concept of rank decomposition introduced in~\cite{botnan-oppermann-oudot}: while free resolutions may not be rank-exact, in the sense that the rank invariant of a module~$M$ may not decompose as the alternating sum of the rank invariants of the terms in a free resolution of~$M$, it is known that the Hilbert function of~$M$ does decompose as such, leading to a notion of signed decomposition of the Hilbert function, which can be viewed as a simpler variant of the rank decomposition.
We show that Hilbert decompositions $\sbetti(-)$ given by the summands in minimal free resolutions are stable in the bottleneck dissimilarity (\cref{main-theorem}) while minimal Hilbert decompositions $\msd(-)$ are not (\cref{instability-minimal-decomposition-example}),
thus answering, in this simpler setting, a question left open in~\cite{botnan-oppermann-oudot}.

In \cite{mccleary2021edit}, McCleary and Patel prove a stability result for a variant of the generalized persistence diagram of a simplicial filtration indexed over a finite lattice.
Their result is stated in terms of an edit distance for filtrations and an edit distance for their variant of the generalized persistence diagram.
Connections between their edit distance for filtrations and the interleaving distance between the corresponding homology persistence modules still remain to be established.
Meanwhile, it is unclear whether their edit distance for generalized persistence diagrams
is non-trivial.

Our notion of signed barcode is closely related to the notion of virtual persistence diagram, introduced by Bubenik and Elchesen \cite{bubenik-elchesen}.
In fact, the space of signed barcodes endowed with the signed $1$-Wasserstein distance is isometric to a certain space of virtual persistence diagrams in the sense of \cite{bubenik-elchesen}.
In contrast, the space of signed barcodes endowed with the bottleneck dissimilarity cannot be directly interpreted as a space of virtual persistence diagrams.
This is a consequence of the fact that our notion of signed barcode allows for the same bar to appear, perhaps multiple times, as both a positive and a negative bar, which is important for proving the bottleneck stability of Betti signed barcodes.

\subsection*{Acknowledgements}
The authors thank Fernando Martin for insightful conversations about homological algebra, Michael Lesnick for useful conversations regarding the presentation distance, and the anonymous reviewers for invaluable feedback that has improved this manuscript.
This work was initiated during the workshop {\em Metrics in Multiparameter Persistence}, organized by Ulrich Bauer, Magnus Botnan, and Michael Lesnick at the Lorentz Center in July 2021.
\ifarxivversion
L.S.~was partially supported by the National Science Foundation through grants CCF-2006661
and CAREER award DMS-1943758.
\fi

\section{Background and notation}
\label{background-section}

We assume familiarity with basic category theory and homological algebra.

\subsection*{Dissimilarities, pseudodistances, and distances}
Recall that a \define{dissimilarity function} on a set $A$ is any symmetric function $A \times A \to [0,\infty]$ whose value is zero on every pair $(a,a)$ for $a\in A$.
An \define{extended pseudodistance} is a dissimilarity function that satisfies the triangle inequality.
An \define{extended distance} on $A$ is an extended pseudodistance $d$ for which $d(a,b) = 0$ implies $a = b$ for all $a,b \in A$.

\subsection*{Multisets}
Let $A$ be a set.
An \define{indexed multiset of elements of $A$} consists of a set $I$ and a function $f : I \to A$.
We denote such a multiset by $(I,f)$.
The \define{cardinality} of $(I,f)$ is simply the cardinality of $I$, and $(I,f)$ is \define{finite} if it has finite cardinality.

Let $(I,f)$ and $(J,g)$ be indexed multisets of elements of $A$.
A \define{bijection} between $(I,f)$ and $(J,g)$ is a bijection $h : I \to J$ without any further restrictions.
We say that $(I,f)$ and $(J,g)$ are \define{isomorphic}, and write $(I,f) = (J,g)$, if there exists a bijection $h : I \to J$ such that $f = g \circ h$.
The \define{union} of $(I,f)$ and $(J,g)$, denoted $I \cup J$, is the indexed multiset of elements of $A$ given by $(I\amalg J, f\amalg g)$, where $I \amalg J$ denotes the disjoint union of sets (i.e., any coproduct in the category of sets) and the function $f\amalg g : I \amalg J \to A$ is $f$ on $I$ and $g$ on $J$.
We say that $(I,f)$ and $(J,g)$ are \define{disjoint} if the images of $f$ and $g$ are disjoint.
We say that $(I,f)$ is \define{contained} in $(J,g)$, and write $I \subseteq J$, if there exists an injective function $h : I \to J$ with $f = g \circ h$.
Moreover, when there is no risk of confusion, we leave the function $f$ of a multiset $(I,f)$ implicit, and abuse notation by not distinguishing an element $i \in I$ from its image $f(i) \in A$.

A different way of defining multisets---more common in the persistence literature---is to encode a multiset of elements of $A$ using a function of the form $A \to \N$, interpreted as a multiplicity function.
It is not hard to see that, for every set $A$, there is a one-to-one correspondence between isomorphism types of finite indexed multisets of $A$, on the one hand, and functions $A \to \N$ of finite support (i.e., which take non-zero values on finitely many elements of $A$), on the other hand.

\subsection*{Indexing posets}
We let $\RR$ denote the poset of real numbers with the standard order, which we interpret as a category with objects real numbers and a unique morphism $r \to s$ whenever $r \leq s \in \RR$.
Throughout the paper, $n \in \N$ denotes a natural number, and $\RR^n$ denotes the product poset (equipped with the product order), which we also interpret as a category.

\subsection*{Barcodes}
An \define{interval} of the poset $\RR^n$ consists of a non-empty subset $\III \subseteq \RR^n$ satisfying the following two properties: If $r,t \in \III$, then $s \in \III$ for all $r \leq s \leq t$; and for all $r,s \in \III$, there exists a finite sequence $t_1, \dots, t_k \in \III$ such that $r = t_1$, $s = t_k$, and $t_m$ and $t_{m+1}$ are comparable for all $1 \leq m \leq k-1$.
An $n$-dimensional \define{interval barcode} is a multiset of intervals of $\RR^n$.
An $n$-dimensional \define{barcode} is a multiset of elements of $\RR^n$, and an $n$-dimensional \define{signed barcode} is an ordered pair of $n$-dimensional barcodes.
When there is no risk of confusion, we refer to $n$-dimensional (signed) barcodes simply as (signed) barcodes.

\subsection*{Persistence modules}
Throughout the paper, $\field$ denotes a field and $\vect$ denotes the category of finite dimensional $\field$-vector spaces.
An \define{$n$-parameter persistence module} is a functor $M : \RR^n \to \vect$.
For $r \leq s \in \RR^n$, we let $\phi^M_{r,s} : M(r) \to M(s)$ denote the structure morphism of $M$.
Given a natural transformation of $n$-parameter persistence modules $f : M \to N$ and $r \in \RR^n$, we denote the $r$-component of $f$ by $f_r : M(r) \to N(r)$.
Depending on the context, we may refer to an $n$-parameter persistence module as a multiparameter persistence module, as a persistence module, or simply as a module.

If $\III \subseteq \RR^n$ is an interval, the \define{interval module} with support $\III$ is the multiparameter persistence module $\field_\III : \RR^n \to \vect$ that takes the value $\field$ on the elements of $\III$ and the value $0$ elsewhere, and is such that all the structure morphisms that can be non-zero are the identity of $\field$. See, e.g., \cref{fig:Betti-matching-intervals} for examples of two-parameter interval modules.
A multiparameter persistence module is \define{interval decomposable} if is isomorphic to a direct sum of interval modules.
Given an interval decomposable multiparameter persistence module $M$, there exists an interval barcode $\B(M)$, unique up to isomorphism of multisets, such that $M \cong \bigoplus_{\III \in \B(M)} \field_\III$.

\subsection*{Free modules and their barcodes}
A multiparameter persistence module is \define{free} if it is isomorphic to a direct sum of modules of the form $F_i$ for $i \in \RR^n$, where $F_i : \RR^n \to \vect$ denotes the interval module with support $\{j \in \RR^n : i \leq j\}$ (see, e.g., the module $M$ of \cref{fig:Betti-matching}).
In particular, free modules are interval decomposable and, as such, they admit an interval barcode.
If a multiparameter persistence module $P$ is free, then the intervals in $\B(P)$ are of the form $\{j \in \RR^n : i \leq j\}$.
For this reason, for a free multiparameter module $P$ we abuse notation and interpret $\B(P)$ as an $n$-dimensional barcode (i.e., a multiset of elements of $\RR^n$) by identifying an interval of the form $\{j \in \RR^n : i \leq j\}$ with its minimum $i \in \RR^n$.
The \define{rank} of a free module $P$ is the cardinality of $\B(P)$.

Recall that, being a category of functors taking values in an abelian category, the category of functors $\RR^n \to \vect$ is an abelian category.
A multiparameter persistence module $M$ is \define{finitely presentable} if it is isomorphic to the cokernel of a morphism between free modules of finite rank.

\subsection*{Interleavings}
We use the notion of multidimensional interleaving for multiparameter persistence modules \cite{lesnick}, which we now recall.
Let $M : \RR^n \to \vect$ and $\epsilon \geq 0 \in \RR^n$.
The \define{$\epsilon$-shift} of $M$, denoted $M[\epsilon] : \RR^n \to \vect$, is defined by $M[\epsilon](r)= M(r+\epsilon)$, with structure maps given by $\phi^M_{r+\epsilon,s+\epsilon} : M[\epsilon](r) \to M[\epsilon](s)$.
Note that $(-)[\epsilon] : \vect^{\RR^n} \to \vect^{\RR^n}$ is a functor, and that there is a natural transformation $\iota_\epsilon^M : M \to M[\epsilon]$ given by the structure maps of $M$.
Let $N : \RR^n \to \vect$.
An \define{$\epsilon$-interleaving} between $M$ and $N$ consists of natural transformations $f : M \to N[\epsilon]$ and $g : N \to M[\epsilon]$ such that $g[\epsilon] \circ f = \iota^M_{2\epsilon}$ and $f[\epsilon] \circ g = \iota^N_{2\epsilon}$.
When $\epsilon \geq 0 \in \RR$, $\epsilon$-interleaving is taken to mean $(\epsilon,\epsilon, \dots,\epsilon)$-interleaving.
The \define{interleaving distance} between persistence modules $M,N : \RR^n \to \vect$ is the extended pseudodistance defined by
\[
    d_I(M,N) = \inf\big\{\,\epsilon \geq 0 \,: \,\text{$M$ and $N$ are $\epsilon$-interleaved }\, \big\} \in \R_{\geq 0}\cup\{\infty\}.
\]
In the proof of \cref{main-corollary}, we use the asymmetric notion of interleaving \cite[Section~2.6.1]{lesnick-thesis} in which, given $\epsilon,\delta \geq 0 \in \RR^n$, we ask the natural transformations of the $(\epsilon,\delta)$-interleaving to be of the form $M \to N[\epsilon]$ and $N \to M[\delta]$, and to compose to $\iota_{\epsilon+\delta}^M$ in one direction and to $\iota_{\epsilon+\delta}^N$ in the other direction.
%

\subsection*{Bottleneck distance}
We now define the bottleneck distance between interval barcodes, as defined in, e.g., \cite{bjerkevik}.
Let $\B$ and $\C$ be ($n$-dimensional) interval barcodes and let $\epsilon \geq 0$.
An \define{$\epsilon$-matching} between $\B$ and $\C$ consists of a bijection $h : \B' \to \C'$ for some $\B' \subseteq \B$ and $\C' \subseteq \C$ such that
\begin{itemize}
    \item
    for all $\III \in \B'$, the interval modules $\field_\III$ and $\field_{h(\III)}$ are $\epsilon$-interleaved;
    \item
    for all $\III \in \B$ not in $\B'$, the interval module $\field_\III$ is $\epsilon$-interleaved with the $0$ module;
    \item
    for all $\JJJ \in \C$ not in $\C'$, the interval module $\field_\JJJ$ is $\epsilon$-interleaved with the $0$ module.
\end{itemize}
The \define{bottleneck distance} between $\B$ and $\C$ is
\[
    d_B(\B,\C) = \inf\big\{\,\epsilon \geq 0 \,: \,\text{there exists an $\epsilon$-matching between $\B$ and $\C$ }\, \big\} \in \R_{\geq 0}\cup\{\infty\}.
\]

Note that, if $\B$ and $\C$ are interval barcodes of free modules, then $d_B(\B,\C)$ coincides with the bottleneck distance defined in the contributions section.
More specifically, note that if $0 \leq \epsilon < \infty$ and all the intervals in $\B$ and $\C$ are of the form $\{j \in \RR^n : i \leq j\}$, then there exists an $\epsilon$-matching between $\B$ and $\C$ if and only if there exists an $\epsilon$-bijection between $\B$ and $\C$ interpreted as $n$-dimensional barcodes, that is, a bijection $h : \B \to \C$ such that $\|i - h(i)\|_\infty \leq \epsilon$ for every $i \in \B$.
Here, as explained above, we are identifying an interval of the form $\{j \in \RR^n : i \leq j\}$ with its minimum $i \in \RR^n$.

\subsection*{Homological algebra of persistence modules}
It has been observed in the persistence literature \cite{carlsson-zomorodian,lesnick-wright,miller2020homological} that the homological algebra of finitely presentable $n$-parameter persistence modules is analogous to that of finitely generated $\N^n$-graded modules over the $\N^n$-graded polynomial ring $\field[x_1,\dots,x_n]$.
For an in-depth exposition about multigraded modules, including the claims made in this section, see \cite{miller-sturmfels}.

A finitely presentable multiparameter persistence module $M : \RR^n \to \vect$ is free if and only if it is projective, in the sense of homological algebra.
Every finitely presentable $M : \RR^n \to \vect$ admits a \define{minimal projective resolution}, that is, a projective resolution $P_\bullet \to M$ with the property that, in each homological degree $k \in \N$, the projective (hence free) module $P_k$ has minimal rank among all possible $k$th terms in a projective resolution of $M$.
Let $P_\bullet \to M$ be a minimal projective resolution of $M$.
Since any two minimal projective resolutions of $M$ are isomorphic, we have that, for every $k \in \N$, the isomorphism type of $P_k$ is independent of the choice of minimal resolution, and is thus an isomorphism invariant of $M$.
Since $P_k$ is free, one can define the \define{multigraded Betti number} of $M$ in homological degree $k$, denoted $\betti_k(M)$, to be the barcode of $P_k$, as follows: $\betti_k(M) := \B(P_k)$.

    Recall that the \define{length} of a projective resolution $P_\bullet \to M$ is the minimum over all $k \in \N$ such that $P_m = 0$ for all $m > k$.
A straightforward consequence of Hilbert's syzygy theorem is that finitely presentable $n$-parameter persistence modules admit projective resolutions of length bounded by the number of parameters $n$:

\begin{theorem}[Hilbert]
    \label{hilberts-theorem}
    Any minimal resolution of a finitely presentable $M : \RR^n \to \vect$ has length at most $n$.
\end{theorem}

\subsection*{Bottleneck stability for free modules}
Free modules satisfy the following bottleneck stability result, due to Bjerkevik.
We remark that the case $n=2$ of the theorem was first established by Botnan and Lesnick in \cite[Corollary~6.6]{botnan-lesnick}.

\begin{theorem}[{\cite[Theorem~4.16]{bjerkevik}}]
    \label{bjerkevik-free-modules}
    Let $M,N : \RR^n \to \vect$ be free modules with $n \geq 2$.
    If $M$ and $N$ are $\epsilon$-interleaved, then there exists an $(n-1)\epsilon$-bijection between $\B(M)$ and $\B(N)$.
\end{theorem}

\section{Stability of Betti signed barcodes}
\label{stability-projective-resolutions-section}

In this section, we prove \cref{stability-resolutions} and \cref{main-theorem}, in this order.
We start with a persistent version of Schanuel's lemma for short exact sequences.

\begin{lemma}
    \label{persistent-schanuel}
    Let $M$ and $N$ be persistence modules, and let $0 \to K \to P \xrightarrow{\alpha} M \to 0$ and $0 \to L \to Q \xrightarrow{\gamma} N \to 0$ be short exact sequences, with $P$ and $Q$ projective.
    If $M$ and $N$ are $\epsilon$-interleaved, then so are $P \oplus L[\epsilon]$ and $K[\epsilon] \oplus Q$.
\end{lemma}
\begin{proof}
    Let $f : M \to N[\epsilon]$ and $g : N \to M[\epsilon]$ form an $\epsilon$-interleaving.
    Define the persistence submodule
    $X \subseteq P \oplus Q[\epsilon]$ by 
    \[
        X(r) = \Big\{ (a,c) \in P(r) \oplus Q(r+\epsilon) : f_r(\alpha_r(a)) = \gamma_{r+\epsilon}(c)\Big\}.
    \]
    Consider the morphism $\pi : X \to P$ with $\pi_r(a,c) = a$ for all $r \in \RR^n$ and $(a,c) \in X(r)$.
    We claim that (1) $\pi$ is surjective, and that (2) the kernel of $\pi$ is isomorphic to $L[\epsilon]$.
Then, since $P$ is projective, it follows from $(1)$ and $(2)$ that $X \cong P \oplus L[\epsilon]$.

\medskip

\noindent Analogously, we can define a persistence module $Y$ with
    \[
        Y(r) = \Big\{ (a,c) \in P(r+\epsilon) \oplus Q(r) : \alpha_{r+\epsilon}(a) = g_r(\gamma_r(c))\Big\}.
        \]
    From a symmetric argument, it follows that $Y \cong K[\epsilon] \oplus Q$.

    \medskip

    We now prove that $X$ and $Y$ are $\epsilon$-interleaved.
    Consider, for each $r \in \RR^n$, the morphism $X(r) \to Y(r+\epsilon)$ that maps
    \[
        P(r) \oplus Q(r+\epsilon) \ni (a,c) \mapsto (\phi^{P}_{r,r+2\epsilon}(a),c) \in P(r+2\epsilon) \oplus Q(r+\epsilon).
    \]
    This is well-defined since, if $(a,c) \in X(r)$, then $f_r(\alpha_r(a)) = \gamma_{r+\epsilon}(c)$ and thus
    \begin{align*}
        \alpha_{r+2\epsilon}(\phi^{P}_{r,r+2\epsilon}(a)) &= \phi^{M}_{r,r+2\epsilon}(\alpha_{r}(a))\\
            &= g_{r+\epsilon}(f_r(\alpha_r(a))) = g_{r+\epsilon}(\gamma_{r+\epsilon}(c)),
    \end{align*}
    so $(\phi^{P}_{r,r+2\epsilon}(a),c) \in Y(r+\epsilon)$ as desired.
    These morphisms assemble into a natural transformation $X \to Y[\epsilon]$. Symmetrically, we define a natural transformation $Y \to X[\epsilon]$ pointwise by $(a,c) \mapsto (a, \phi^{Q}_{r,r+2\epsilon}(c))$.    The fact that these natural transformations form an $\epsilon$-interleaving between $X$ and $Y$ is immediate.

    \medskip
    
    To conclude the proof, we prove the claims $(1)$ and $(2)$, starting with $(1)$.
    Let $r \in \RR^n$ and $a \in P(r)$.
    Since $\gamma_{r+\epsilon} : Q(r+\epsilon) \to N(r+\epsilon)$ is surjective, there exists $c \in Q(r+\epsilon)$ such that $\gamma_{r+\epsilon}(c) = f_r(\alpha_r(a))$.
    It follows that $(a,c) \in X(r)$, which proves the claim.

    For $(2)$, note that, for $r \in \RR^n$, the kernel of $\pi_r$ is, by definition,
    \[
        \Big\{ (a,c) \in P(r) \oplus Q(r+\epsilon) : f_r(\alpha_r(a)) = \gamma_{r+\epsilon}(c), a = 0\Big\} \subseteq X(r).
    \]
    This is naturally isomorphic to $\{ c \in Q(r+\epsilon) : 0 = \gamma_{r+\epsilon}(c)\}$, which is naturally isomorphic to $L(r+\epsilon)$, concluding the proof.
\end{proof}

We now move on to our persistent version of Schanuel's lemma for projective resolutions:

\stabilityresolutions*
\begin{proof}


    We may assume that the projective resolutions $P_\bullet \xrightarrow{\alpha} M$ and $Q_\bullet \xrightarrow{\gamma} N$ have lengths that are bounded above by $\ell \in \N$, and proceed by induction on $\ell$.

    The case $\ell = 0$ is immediate.

    For $\ell \geq 1$, consider the short exact sequences
    $0 \to \ker \alpha \to P_0 \to M \to 0$ and
    $0 \to \ker \gamma \to Q_0 \to N \to 0$.
    By \cref{persistent-schanuel}, we have that 
    $(\ker \gamma)[\epsilon] \oplus P_0$ and $(\ker \alpha)[\epsilon] \oplus Q_0$ are $\epsilon$-interleaved.
    We can then use the inductive hypothesis on the projective resolutions
    \begin{align*}
    0 &\to P_\ell[\epsilon] \to \dots \to P_2[\epsilon] \to P_1[\epsilon] \oplus Q_0 \to (\ker \alpha)[\epsilon] \oplus Q_0 \;\;\;\text{and}\\
    0 &\to Q_\ell[\epsilon] \to \dots \to Q_2[\epsilon] \to Q_1[\epsilon] \oplus P_0 \to (\ker \gamma)[\epsilon] \oplus P_0
    \end{align*}
    of $(\ker \alpha)[\epsilon] \oplus Q_0$ and $(\ker \gamma)[\epsilon] \oplus P_0$, respectively, concluding the proof.
\end{proof}

\begin{corollary}
    \label{main-corollary}
    Let $M$ and $N$ be persistence modules.
    Let $P_\bullet \to M$ and $Q_\bullet \to N$ be projective resolutions of length at most $\ell$.
    If $M$ and $N$ are $\epsilon$-interleaved, then $\bigoplus_{i \in \N} P_{2i} \oplus \bigoplus_{i \in \N} Q_{2i+1}$ and $\bigoplus_{i \in \N} P_{2i+1} \oplus \bigoplus_{i \in \N} Q_{2i}$ are $(\ell+1)\epsilon$-interleaved.
\end{corollary}
\begin{proof}
    In this proof, we use the asymmetric version of interleaving.
    Note that $\bigoplus_{i \in \N} P_{2i} \oplus \bigoplus_{i \in \N} Q_{2i+1}$ is  $(0,\ell\epsilon)$-interleaved with $\bigoplus_{i \in \N} P_{2i}[2i\epsilon] \oplus \bigoplus_{i \in \N} Q_{2i+1}[(2i+1)\epsilon]$, as any persistence module is $(0,\delta)$-interleaved with its $\delta$-shift.
    Symmetrically, $\bigoplus_{i \in \N} P_{2i+1}[(2i+1)\epsilon] \oplus \bigoplus_{i \in \N} Q_{2i}[2i\epsilon]$ is  $(\ell\epsilon,0)$-interleaved with $\bigoplus_{i \in \N} P_{2i+1} \oplus \bigoplus_{i \in \N} Q_{2i}$. The result follows then from  \cref{stability-resolutions} and the fact that  a composite of an $(\epsilon_1,\epsilon_2)$-interleaving and a $(\delta_1,\delta_2)$-interleaving is an $(\epsilon_1 + \delta_1, \epsilon_2 + \delta_2)$-interleaving.
\end{proof}

\maintheorem*
\begin{proof}
    Let $M,N : \RR^n \to \vect$.
    It is enough to show that if $M$ and $N$ are $\epsilon$-interleaved for some $\epsilon \geq 0$, then there is an $(n+1)(n-1)\epsilon$-bijection between $\betti_{\even}(M) \cup \betti_{\odd}(N)$ and $\betti_{\odd}(M) \cup \betti_{\even}(N)$.

    Let $P_\bullet \to M$ and $Q_\bullet \to N$ be minimal projective resolutions.
    By \cref{hilberts-theorem}, they have length at most $n$, so it follows from \cref{main-corollary} that $\bigoplus_{i \in \N} P_{2i} \oplus \bigoplus_{i \in \N} Q_{2i+1}$ and $\bigoplus_{i \in \N} P_{2i+1} \oplus \bigoplus_{i \in \N} Q_{2i}$
    are $(n+1)\epsilon$-interleaved.
    Note that, as multisets of elements of $\RR^n$, we have
    \begin{align*}
       \B\left(\bigoplus_{i \in \N} P_{2i} \oplus \bigoplus_{i \in \N} Q_{2i+1}\right) &= \betti_\even(M) \cup \betti_\odd(N) \;\;\;\; \text{and}\\[0.5em]
        \B\left(\bigoplus_{i \in \N} P_{2i+1} \oplus \bigoplus_{i \in \N} Q_{2i}\right) &= \betti_\odd(M) \cup \betti_\even(N),
    \end{align*}
    by definition of the multigraded Betti numbers.
    Then, \cref{bjerkevik-free-modules} implies that there exists an $(n+1)(n-1)\epsilon$-bijection between $\betti_\even(M) \cup \betti_\odd(N)$ and 
    $\betti_\odd(M) \cup \betti_\even(N)$, as required.
\end{proof}

\begin{remark}\label{stability-n=1}
When $n=1$, the usual isometry theorem for persistence barcodes implies that $\widehat{d_B}\big(\sbetti(M),\, \sbetti(N)\big)\, \leq\, 2\cdot d_I(M,N)$, since any partial matching  between persistence barcodes yields a bijection between the corresponding Betti signed barcodes. The bottleneck cost of this bijection is at most twice that of the partial matching itself, because the cost of matching the endpoints of a pair of matched intervals together is the same as that of matching the intervals themselves, while the  cost of matching the endpoints of an unmatched interval together is the length of the interval and not half the length.
\end{remark}

\section{Universality of the bottleneck dissimilarity on Betti signed barcodes}
\label{weak-universality-section}

In this section, we prove the universality result \cref{weak-universality}, and we give \cref{possible-issue-example}, which shows that the dissimilarity function on signed barcodes $\widehat{d_B}$ does not satisfy the triangle inequality, even when restricted to Betti signed barcodes of finitely presentable persistence modules.
As a consequence of \cref{weak-universality}, we prove \cref{no-go-thm}, a no-go result stating that there are no non-trivial stable and balanced dissimilarities on Betti signed barcodes that satisfy the triangle inequality.

\weakuniversality*
\begin{proof}
    To see that $\dmax$ exists, note that it can be defined by
    \[
        \dmax(\B,\C) = \sup \left\{ \widehat{d}(\B,\C) : \widehat{d} \text{ balanced and stable dissimilarity on finite signed barcodes }\right\}.
    \]

    Let $\widehat{d}$ be a dissimilarity function on finite signed barcodes that is balanced and stable.
    We must show that $\widehat{d}(\B,\C) \leq \widehat{d_B}(\B,\C)$.
    Fix finite signed barcodes $\B$ and $\C$ and consider the free modules $P = \bigoplus_{i \in \B_+ \cup \C_-} F_i$ and $Q = \bigoplus_{j \in \C_+ \cup \B_-} F_j$.
    We can now compute as follows:
    \begin{align*}
        \widehat{d}((\B_+, \B_-), (\C_+, \C_-)) &\stackrel{Eq.~\eqref{balanced-dissimilarity}}{=} \widehat{d}\big((\B_+ \cup \C_-,\emptyset), (\C_+ \cup \B_-, \emptyset)\big)  = \widehat{d}(\sbetti(P),\sbetti(Q))\\
        &\leq d_I(P,Q) \leq d_B(\B(P),\B(Q)) = \widehat{d_B}((\B_+, \B_-), (\C_+, \C_-)).
    \end{align*}
    
    To prove that $\widehat{d_B}(\B,\C)/(n^2-1)\leq\dmax(\B,\C)$ when $n\geq 2$, note that the dissimilarity function $\widehat{d_B}/(n^2-1)$ is balanced, and that it is stable by \cref{main-theorem}. A similar argument shows that $\widehat{d_B}(\B,\C)/2\leq\dmax(\B,\C)$ when $n=1$.
\end{proof}

\begin{example}
    \label{possible-issue-example}
    Note, on the one hand, that $\widehat{d_B}(\sbetti(F_i),\sbetti(F_j)) = \|i - j\|_\infty$ for all $i, j\in\RR^n$.
    On the other hand, fix $k \geq 1 \in \N$ and consider $F_{(0,1)}, F_{(1,0)}, A_k : \RR^2 \to \vect$, with $A_k$ the interval module with support the set $\{x \in \RR^2 : \exists\, 0 \leq m \leq k \text{ such that } x \geq (m/k,1-m/k)\}$ (see \cref{possible-issue-example-figure} for an illustration).
    We have
    \[
        \sbetti(A_k) = \Big(\big\{(m/k,1-m/k)\big\}_{0 \leq m \leq k}  \,,\,\big\{((m+1)/k,1-m/k)\big\}_{0 \leq m \leq k-1}\Big).
    \]
    Then, $\widehat{d_B}(\sbetti(F_{(0,1)}), \sbetti(A_k)) \to 0$ and $\widehat{d_B}(\sbetti(A_k), \sbetti(F_{(1,0)})) \to 0$ while $\widehat{d_B}(\sbetti(F_{(0,1)}), \sbetti(F_{(1,0)}))$ remains equal to~$1$ as $k \to \infty$.
\end{example}

\nogothm*
\begin{proof}
    Let $M,N : \RR^n \to \vect$ be finitely presentable with $d_I(M,N) < \infty$.
    Since $\widehat{d}$ satisfies the triangle inequality, it is enough to show that, for every $\epsilon > 0$, there exists a finitely presented module $A : \RR^n \to \vect$ such that $\widehat{d}(\sbetti(M),\sbetti(A)) \leq \epsilon$ and $\widehat{d}(\sbetti(N), \sbetti(A)) \leq \epsilon$.

    Since $d_I(M,N) < \infty$, by \cref{main-theorem}, there exists a $\delta$-bijection $h : \betti_{2\N}(M) \cup \betti_{2\N+1}(N) \to \betti_{2\N}(N) \cup \betti_{2\N+1}(M)$ for some $\delta < \infty$.
    Given $i \in \betti_{2\N}(M) \cup \betti_{2\N+1}(N)$, consider a finite sequence $p^i_1, \dots, p^i_{k_i} \in \RR^n$ such that $p^i_1 = i$, $p^i_{k_i} = h(i)$, and $\|p^i_m - p^i_{m+1}\|_\infty \leq \epsilon/2$ for all $1 \leq m \leq {k_i}-1$.
    For $a < b \in \RR^n$, let $L_{a,b} : \RR^n \to \vect$ denote the interval module with support $\{x \in \RR^n : a \leq x \text{ and } b \nleqslant x\}$ (see \cref{instability-minimal-decomposition-figure} for an illustration in the two-parameter case).

    For notational convenience, in the rest of this proof, if $r \in \RR^n$ and $a \in \RR$, we let $r + a = (r_1 + a, \dots, r_n + a) \in \RR^n$.
    Define the module $A = M \oplus B$, where
    \[
        B = \bigoplus_{i \in \betti_{2\N}(M) \cup \betti_{2\N+1}(N)} \;\; \bigoplus_{1 \leq m \leq k_i}\;\; L_{p^i_m,p^i_m+\epsilon/2}.
    \]
    Note that $d_I(B,0) \leq \epsilon/4$, so, since $\widehat{d}$ is stable, we have $\widehat{d}(\sbetti(A),\sbetti(M)) \leq d_I(A,M) \leq d_I(B,0) \leq \epsilon/4 \leq \epsilon$, by definition of $A$.
    It remains to be shown that $\widehat{d}(\sbetti(A),\sbetti(N)) \leq \epsilon$.
    By \cref{weak-universality}, it is sufficient to prove that $\widehat{d_B}(\sbetti(A),\sbetti(N)) \leq \epsilon$.
    In order to prove this inequality, note that $\betti_{2\N}(A) = \betti_{2\N}(M) \cup \betti_{2\N}(B)$, $\betti_{2\N+1}(A) = \betti_{2\N+1}(M) \cup \betti_{2\N+1}(B)$, and that
    \begin{align*}
        \betti_{2\N}(B) &= \betti_0(B) = \{p^i_m\}_{i \in \betti_{2\N}(M) \cup \betti_{2\N+1}(N),\; 1 \leq m \leq k_i}, \;\;\text{and}\\
        \betti_{2\N+1}(B) &= \betti_1(B) = \{p^i_m+\epsilon/2\}_{i \in \betti_{2\N}(M) \cup \betti_{2\N+1}(N),\; 1 \leq m \leq k_i},
    \end{align*}
    so we can construct an $\epsilon$-bijection between the following multisets:
    \begin{align*}
        \betti_{2\N}(A) \cup \betti_{2\N+1}(N) &= \betti_{2\N}(M) \cup \betti_{2\N}(B) \cup \betti_{2\N+1}(N),\\
        \betti_{2\N}(N) \cup \betti_{2\N+1}(A) &= \betti_{2\N}(N) \cup \betti_{2\N+1}(M) \cup \betti_{2\N+1}(B)
    \end{align*}
    by matching as follows:
    match $i \in \betti_{2\N}(M) \cup \betti_{2\N+1}(N)$ to $p^i_{1} + \epsilon/2 \in \betti_{2\N+1}(B)$;
    match $p^i_m \in \betti_{2\N}(B)$ to $p^{i+1}_{m} + \epsilon/2 \in \betti_{2\N+1}(B)$ for $1 \leq m \leq k_i-1$;
    and match $p^i_{k_i} \in \betti_{2\N}(B)$ to $h(i) \in \betti_{2\N}(N) \cup \betti_{2\N+1}(M)$.
\end{proof}

\section{Existence, uniqueness, and bottleneck instability of minimal Hilbert decompositions}
\label{instability-section}

In this section, we show that, for finitely presentable modules, minimal Hilbert decompositions exist and are unique (\cref{existence-minimal-hilbert}), and we give \cref{instability-minimal-decomposition-example}, which shows that, when compared using $\widehat{d_B}$, minimal Hilbert decompositions are not stable.

Recall that a finite Hilbert decomposition of a function $\eta : \RR^n \to \Z$ consists of a pair $(P,Q)$ of free modules of finite rank such that $\eta = \hil(P) - \hil(Q)$, and that such a decomposition is \define{minimal} if $\B(P)$ and $\B(Q)$ are disjoint as multisets.


\begin{lemma}
    \label{hilbert-complete-free}
    Let $P$ and $Q$ be free multiparameter persistence modules of finite rank.
    If $\hil(P) = \hil(Q)$, then $P \cong Q$.
\end{lemma}
\begin{proof}
    We prove that $\B(P) = \B(Q)$ by induction on the rank of $P$.
    The base case is immediate.
    For the inductive step, let $i_0 \in \RR^n$ be minimal with the property that $\hil(P)(i_0) \neq 0$, which must exist since $P$ is of finite rank.
    We claim that $i_0 \in \B(P)$.
    In order to see this, recall that $P \cong \bigoplus_{i \in \B(P)} F_i$ and thus $\hil(P) = \sum_{i \in \B(P)} \hil(F_i)$.
    Since $\hil(P)(i_0) \neq 0$, there must exist $i \in \B(P)$ with $i \leq i_0$, and in fact we have $i = i_0$, by minimality of $i_0$.
    The same argument shows that $i_0 \in \B(Q)$.
    This means that, without loss of generality, we can take $F_{i_0}$ to be a summand of both $P$ and $Q$.
    We can then apply the inductive hypothesis to $P/F_{i_0}$ and $Q/F_{i_0}$.
%
\end{proof}

\begin{proposition}
    \label{existence-minimal-hilbert}
    Let $M : \RR^n \to \vect$ be finitely presentable.
    Then, the function $\hil(M)$ admits a minimal Hilbert decomposition $(P^*,Q^*)$ and, for every other minimal Hilbert decomposition $(P,Q)$ of $\hil(M)$, we have $P \cong P^*$ and $Q \cong Q^*$.
\end{proposition}
\begin{proof}
    We start by proving existence.
    As noted in the introduction, any finite projective resolution $P_\bullet \to M$ induces a Hilbert decomposition $\left(\bigoplus_{k \in \even} P_{k},\bigoplus_{k \in \odd} P_k\right)$ of $\hil(M)$.
    Since $M$ does admit some finite projective resolution, for example its minimal resolution (\cref{hilberts-theorem}), it follows that a minimal Hilbert decomposition of $M$ can be obtained by:
    \begin{enumerate}
    \item considering $(\betti_{2\N}(M),\betti_{2\N+1}(M))$, then
    \item cancelling (with multiplicity) all the bars that appear in both $\betti_{2\N}(M)$ and $\betti_{2\N+1}(M)$, to obtain a signed barcode $(\B,\C)$, and finally
      \item constructing the Hilbert decomposition $(\bigoplus_{i \in \B} F_i, \bigoplus_{j \in \C} F_j)$.
\end{enumerate}
        This Hilbert decomposition is minimal by construction. 

        In order to prove uniqueness, note that, given $(P^*,Q^*)$ and $(P,Q)$ as in the statement, we have $\hil(P^*) - \hil(Q^*) = \hil(P) - \hil(Q)$, which implies
        \[\hil(P^* \oplus Q) = \hil(P^*) + \hil(Q) = \hil(P) + \hil(Q^*) = \hil(P\oplus Q^*),\]
        and thus $P^* \oplus Q \cong P \oplus Q^*$ by \cref{hilbert-complete-free}.
    It follows that $\B(P^*) \cup \B(Q) = \B(P) \cup \B(Q^*)$, and so that $\B(P^*) = \B(P)$ and $\B(Q^*) = \B(Q)$ by minimality.
\end{proof}

\begin{remark}
    \label{remark-same-hilbert-function}
The argument in the existence part of the proof of \cref{existence-minimal-hilbert} shows that any Hilbert decomposition can be made minimal by cancelling (with multiplicity) the summands corresponding to bars that appear in the positive and in the negative parts.
\end{remark}

Recall that $\msd(M) = \left(\B(P^*), \B(Q^*)\right)$ denotes the signed barcode coming from a minimal Hilbert decomposition~$(P^*, Q^*)$ of $\hil(M)$.
The following example implies that there is no constant $c \geq 1$ such that $\widehat{d_B}(\msd(M), \msd(N)) \leq \, c\cdot d_I(M,N)$.

\begin{figure}
\centering
\begin{minipage}{0.45\textwidth}
    \centering
\def\svgwidth{.42\textwidth}
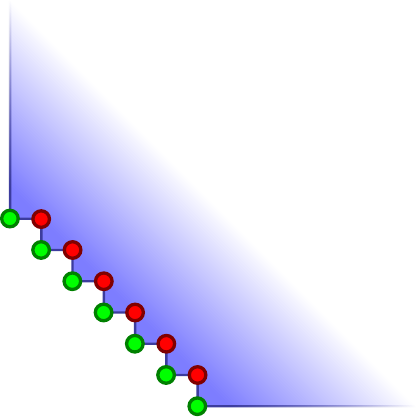
\caption{The module $A_k$ of \cref{possible-issue-example}, when $k = 6$.}
\label{possible-issue-example-figure}
\end{minipage}\hfill
\begin{minipage}{0.45\textwidth}
    \centering
{\def\svgwidth{.4\textwidth}
{
\begingroup%
  \makeatletter%
  \providecommand\color[2][]{%
    \errmessage{(Inkscape) Color is used for the text in Inkscape, but the package 'color.sty' is not loaded}%
    \renewcommand\color[2][]{}%
  }%
  \providecommand\transparent[1]{%
    \errmessage{(Inkscape) Transparency is used (non-zero) for the text in Inkscape, but the package 'transparent.sty' is not loaded}%
    \renewcommand\transparent[1]{}%
  }%
  \providecommand\rotatebox[2]{#2}%
  \newcommand*\fsize{\dimexpr\f@size pt\relax}%
  \newcommand*\lineheight[1]{\fontsize{\fsize}{#1\fsize}\selectfont}%
  \ifx\svgwidth\undefined%
    \setlength{\unitlength}{141.46173135bp}%
    \ifx\svgscale\undefined%
      \relax%
    \else%
      \setlength{\unitlength}{\unitlength * \real{\svgscale}}%
    \fi%
  \else%
    \setlength{\unitlength}{\svgwidth}%
  \fi%
  \global\let\svgwidth\undefined%
  \global\let\svgscale\undefined%
  \makeatother%
  \begin{picture}(1,1.02169726)%
    \lineheight{1}%
    \setlength\tabcolsep{0pt}%
    \put(0,0){\includegraphics[width=\unitlength,page=1]{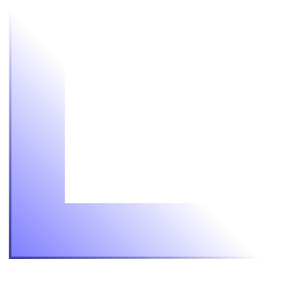}}%
    \put(0.00555859,0.01422091){\makebox(0,0)[lt]{\lineheight{1.25}\smash{\begin{tabular}[t]{l}$a$\end{tabular}}}}%
    \put(0.26134119,0.41680718){\makebox(0,0)[lt]{\lineheight{1.25}\smash{\begin{tabular}[t]{l}$b$\end{tabular}}}}%
    \put(0,0){\includegraphics[width=\unitlength,page=2]{example-betti-matching-3.pdf}}%
  \end{picture}%
\endgroup%
}}
\caption{The module $L_{a,b}$ of \cref{instability-minimal-decomposition-example} and the proof of \cref{no-go-thm}, when $n=2$.}
\label{instability-minimal-decomposition-figure}
\end{minipage}
\end{figure}

\begin{example}
    \label{instability-minimal-decomposition-example}
    Recall that, for $a < b \in \RR^n$, we let $L_{a,b} : \RR^n \to \vect$ denote the interval module with support $\{x \in \RR^n : a \leq x \text{ and } b \nleqslant x\}$, and that we have $\sbetti(L_{a,b}) = (\{a\},\{b\})$.
    Let $\epsilon > 0$ and $m \in \N$, and consider the two-parameter persistence module
    \[
        M = L_{(0,0),(\epsilon,\epsilon)} \oplus L_{(\epsilon,\epsilon),(2\epsilon,2\epsilon)} \oplus \cdots \oplus L_{((m-1)\epsilon,(m-1)\epsilon),(m\epsilon,m\epsilon)}.
    \]
    The module $M$ is $\epsilon/2$-interleaved with the zero module $0$.
    But $\msd(M) = \left(\{(0,0)\},\{(m\epsilon,m\epsilon)\}\right)$, so the optimal matching between $\msd(M)$ and $\msd(0) = (\emptyset, \emptyset)$ has cost $m\epsilon$, which can be made arbitrarily large by increasing $m$. Thus, by cancelling out all the bars that are common to  the positive and negative parts of the signed barcode of~$M$, the minimal Hilbert decomposition makes it impossible to build low-cost matchings with signed barcodes coming from certain nearby persistence modules like the zero module. By contrast, the Betti signed barcode keeps enough common bars (typically  the two copies of each point $(k\epsilon, k\epsilon)$ for $k=1$ to~$m-1$ in this example) to allow for low-cost matchings, according to \cref{main-theorem}.   
\end{example}

\section{$1$-Wasserstein stability of minimal Hilbert decompositions}
\label{stability-hilbert-decompositions}

In this section, we prove \cref{wasserstein-stability}.
The section is structured as follows.
In \cref{signed-wasserstein-distance-subsection} we define the signed $1$-Wasserstein distance $\widehat{\dwass}$ and establish some of its main properties.
In \cref{presentation-distance-subsection} we recall the definition of the $1$-presentation distance $\dIone$ introduced by Bjerkevik and Lesnick in \cite{bjerkevik-lesnick}.
In \cref{proof-of-wasserstein-stability-subsection} we recall necessary background from \cite{bjerkevik-lesnick} and prove \cref{wasserstein-stability}.
The key results from \cite{bjerkevik-lesnick} that we use are recalled below as \cref{useful-universality-lp} and \cref{grid-lemma}.

\subsection{The signed $1$-Wasserstein distance}
\label{signed-wasserstein-distance-subsection}
Let $\B$ and $\C$ be finite barcodes.
Let $h : \B \to \C$ be a bijection of multisets of elements of $\RR^n$.
The \define{$1$-Wasserstein cost} of $h$ is $1\text{-}\cost(h) = \sum_{i \in \B} \|i - h(i)\|_1$.
Define the following extended pseudodistance on finite (unsigned) barcodes
\[
    \dwass(\B,\C) = \inf\big\{\,\epsilon \geq 0 \,: \,\exists\text{ bijection }\, h : \B \to \C \text{ with } 1\text{-}\cost(h)\leq \epsilon\,\big\} \in \R_{\geq 0}\cup\{\infty\}.
\]

\begin{definition}
    \label{signed-wass-def}
    The \define{signed $1$-Wasserstein distance} between finite signed barcodes $\B = (\B_+,\B_-)$ and $\C = (\C_+,\C_-)$ is given by 
    \[
        \widehat{\dwass}(\B,\C) = \dwass\big(\B_+ \cup \C_-,\, \C_+ \cup \B_-\big).
    \]
\end{definition}

For context, we mention that there is a natural way in which finite signed barcodes induce finite, signed Radon measures, and this happens in such a way that the signed $1$-Wasserstein distance between signed barcodes corresponds to the Kantorovich norm of the difference between their corresponding signed measures; see \cref{proposition:wasserstein-is-kantorovich} for a precise statement.

A finite signed barcode $(\B_+,\B_-)$ is \define{reduced} if $\B_+$ and $\B_-$ are disjoint.
One can turn any finite signed barcode $\B = (\B_+,\B_-)$ into a reduced signed barcode  $\overline{\B} = (\overline{\B}_+, \overline{\B}_-)$ by cancelling the common bars in $\B_+$ and $\B_-$ with multiplicity. Then,  $\overline{\B}$ is included in~$\B$ in the sense that $\overline{\B}_+ \subseteq \B_+$ and $\overline{\B}_- \subseteq \B_-$, and  $\overline{\B}$ is in fact maximal with respect to inclusion among all the reduced signed barcodes that are included in $\B$.
The reduced barcode $\overline{\B}$ can also be described as the minimal barcode included in $\B$ with the property that
\[
    \sum_{i \in \B_+} \hil(F_i) - \sum_{j \in \B_-} \hil(F_j) = 
    \sum_{i \in \overline{\B}_+} \hil(F_i) - \sum_{j \in \overline{\B}_-} \hil(F_j).
\]
The signed barcode $\overline{\B}$ is unique up to isomorphism of multisets.
Note that, with this definition, we have $\overline{\sbetti(M)} = \msd(M)$ for any finitely presentable persistence module $M\colon \RR^n \to \vect$.

\begin{proposition}
    \label{properties-signed-wass}
    \begin{enumerate}
        \item For any pair of finite signed barcodes $\B$ and $\C$, we have $\widehat{\dwass}(\B,\C) = 0$ if and only if $\overline{\B}=\overline{\C}$.
        \item The dissimilarity $\widehat{\dwass}$ satisfies the triangle inequality.
        \item For any pair of finite signed barcodes $\B$ and $\C$, we have $\widehat{\dwass}(\B,\C) = \widehat{\dwass}(\overline{\B},\overline{\C})$.
    \end{enumerate}
\end{proposition}
\begin{proof}
    We start with the first claim.
    Let $\B$ and $\C$ be finite signed barcodes.
    It is clear that $\overline{\B} = \overline{\C}$ implies $\widehat{\dwass}(\B,\C) = 0$.
    For the converse, we may assume given a bijection $h : \B_+ \cup \C_- \to \C_+ \cup \B_-$ with $1\text{-}\cost(h) = 0$.
    Let $\B'_+ \subseteq \B_+$ be the multiset of elements of $\RR^n$ given by the elements of $\B_+$ that $h$ maps to some element of $\C_+$.
    Similarly, define $\C'_- \subseteq \C_-$ as the multiset of elements that $h$ maps to $\B_-$.
    Analogously, define $\C'_+$ and $\B'_-$ using the inverse of the bijection $h$.
    It is then clear that $h$ restricts to bijections $\B'_+ \to \C'_+$ and $\C'_- \to \B'_-$ of cost $0$.
    Finally, since $1\text{-}\cost(h) = 0$, we have $\B_+ \setminus \B_+' = \B_- \setminus \B_-'$, hence $\overline{(\B'_+,\B'_-)} = \overline{\B}$ and $\overline{(\C'_+,\C'_-)} = \overline{\C}$.

    For the second claim, let $(\B_+,\B_-)$, $(\C_+,\C_-)$, and $(\D_+,\D_-)$ be finite signed barcodes, and let $h : \B_+ \cup \C_- \to \C_+ \cup \B_-$ and $g : \C_+ \cup \D_- \to \D_+ \cup \C_-$ be bijections with $1\text{-}\cost(h) = \epsilon$ and $1\text{-}\cost(g) = \delta$.
    It is enough to prove that there exists a bijection $\B_+ \cup \D_- \to \D_+ \cup \B_-$ of cost at most $\epsilon + \delta$.
    Such a bijection can be defined constructively, as follows.
    Given any $i \in \B_+ \cup \C_- \cup \C_+ \cup \D_-$, let
    \[
        s(i) = \begin{cases}
                h(i), &\text{ if $i \in \B_+ \cup \C_-$},\\
                g(i), &\text{ if $i \in \C_+ \cup \D_-$}. 
                \end{cases}
    \]
    We claim that, starting from $i \in \B_+ \cup \D_-$, one can apply $s$ repeatedly until, after finitely many applications, one gets an element $s^\ast(i) \in \D_+ \cup \B_-$, and that $s^\ast : \B_+ \cup \D_- \to \D_+ \cup \B_-$ is a bijection of cost at most $\epsilon + \delta$.
    In order to see this, consider the directed graph $G$ with vertices given by $\B_+ \cup \B_- \cup \C_+ \cup \C_- \cup \D_+ \cup \D_-$ and a directed edge $i \to j$ if and only if $j = s(i)$.
    Note that the vertices in $\B_+ \cup \D_-$ have in-degree $0$ and out-degree $1$, the vertices in $\B_- \cup \D_+$ have in-degree $1$ and out-degree $0$, and the vertices in $\C_+ \cup \C_-$ have in-degree $1$ and out-degree $1$.
    This implies that $s^\ast$ is well-defined and injective.
    An analogous argument using the inverses of $h$ and $g$ shows that $s^\ast$ is surjective, and thus a bijection.
    Finally, the cost of $s^\ast$ is at most $\epsilon + \delta$ since each edge of $G$ belongs to at most one path from an $i \in \B_+ \cup \D_-$ to $s^\ast(i) \in \D_+ \cup \B_-$.

    The third claim follows from the triangle inequality of $\widehat{\dwass}$ (second claim) and the fact that $\widehat{\dwass}(\B,\overline{\B}) = 0$ for any finite signed barcode $\B$ (first claim).
\end{proof}

\begin{corollary}
    \label{corollary:same-hilbert-function}
    Let $M, N : \RR^n \to \vect$ be finitely presentable.
    The following are equivalent:
    \begin{itemize}
            \item $\hil(M) = \hil(N)$;
            \item $\msd(M) = \msd(N)$ as signed barcodes;
            \item $\widehat{d_B}(\sbetti(M),\sbetti(N)) = 0$.
    \end{itemize}
\end{corollary}
\begin{proof}
    If $\hil(M) = \hil(N)$, then $\msd(M) = \msd(N)$ as signed barcodes by \cref{existence-minimal-hilbert}.

    If $\msd(M) = \msd(N)$, then $\widehat{\dwass}(\sbetti(M),\sbetti(N)) = \widehat{\dwass}(\msd(M),\msd(N)) = 0$,
    by \cref{properties-signed-wass}(3.) and the fact that $\overline{\sbetti(M)} = \msd(M)$ and $\overline{\sbetti(N)} = \msd(N)$.
    It follows that $\widehat{d_B}(\sbetti(M),\sbetti(N)) = 0$.
    This is because we always have $\widehat{d_B}(\sbetti(M),\sbetti(N)) \leq \widehat{\dwass}(\sbetti(M),\sbetti(N))$, since, for every $\epsilon \geq 0$, an $\epsilon$-bijection has $1$-Wasserstein cost at least $\epsilon$.
    
    Finally, if $\widehat{d_B}(\sbetti(M),\sbetti(N)) = 0$, then, since $\sbetti(M)$ and $\sbetti(N)$ are finite, there exists a $0$-bijection $\betti_{2\N}(M) \cup \betti_{2\N+1}(N) \to \betti_{2\N}(N) \cup \betti_{2\N+1}(M)$, and thus $\widehat{\dwass}(\sbetti(M),\sbetti(N)) = 0$.
    It follows from \cref{properties-signed-wass}(1.) that $\msd(M) = \msd(N)$ as signed barcodes, and thus that $\hil(M) = \hil(N)$.
\end{proof}

\begin{remark}
    \label{triangle-inequality-proof-for-bottleneck}
    It is worth pointing out why the proof of the triangle inequality for the signed $1$-Wasserstein distance does not work in the case of the bottleneck dissimilarity.
    Suppose one follows an argument analogous to the one in the proof of \cref{properties-signed-wass}(2.), using $\widehat{d_B}$ instead of $\widehat{\dwass}$. With the notation of the proof, given an $\epsilon$-bijection $h : \B_+ \cup \C_- \to \C_+ \cup \B_-$ and a $\delta$-bijection $g : \C_+ \cup \D_- \to \D_+ \cup \C_-$, one constructs a bijection $s^\ast : \B_+ \cup \D_- \to \D_+ \cup \B_-$.
    The problem one encounters is that $s^\ast$ may not be an $(\epsilon+\delta)$-bijection, for the following reason.
    Given $i \in \B_+$, the directed path from $i$ to $s^\ast(i) \in \D_+ \cup \B_-$ in the graph $G$
    may contain strictly more than two edges, for a total length strictly larger than $\epsilon + \delta$.
    An analogous problem arises when $i \in \D_-$.

    This is exactly what happens in cases such as the one of \cref{possible-issue-example}: in that case, taking $(\B_+, \B_-) = \sbetti(F_{(0,1)})$, $(\C_+, \C_-) = \sbetti(A_k)$, and $(\D_+,\D_-) = \sbetti(F_{(1,0)})$, we have
    \begin{align*}
        \B_+ &=  \{(0,1)\}\\
        \B_- &=  \emptyset\\
        \C_+ &=  \big\{(m/k,1-m/k)\big\}_{0 \leq m \leq k}\\
        \C_- &=  \big\{((m+1)/k,1-m/k)\big\}_{0 \leq m \leq k-1}\\
        \D_+ &=  \{(1,0)\}\\
        \D_- &=  \emptyset
    \end{align*}
    and, although there are $1/k$-bijections $h : \B_+ \cup \C_- \to \C_+ \cup \B_-$ and $g : \C_+ \cup \D_- \to \D_+ \cup \C_-$, there is a unique possible composite bijection $s^* : \B_+ = \B_+ \cup \D_- \to \D_+ \cup \B_- = \D_+$, which is always a $1$-bijection independently of the value of $k \in \N$.
\end{remark}

\subsection{The $1$-presentation distance}
\label{presentation-distance-subsection}

For further details about the notions introduced in this section, we refer the reader to \cite{bjerkevik-lesnick}, where the presentation distances were first introduced.

We start with the notion of presentation matrix; the point of this definition is to give a concrete encoding of a presentation of a finitely presentable module.
As such, a presentation matrix specifies a morphism between free $n$-parameter persistence modules; this is done with an actual matrix together with $\RR^n$-valued labels for the rows and columns of the matrix.
More precisely, a \define{presentation matrix} $P$ consists of numbers $r,c \in \N$, a matrix $\matt \in \field^{r \times c}$, and a function $\LLL(P) : \{1, \dots, r + c\} \to \RR^n$ such that, for any non-zero matrix coefficient $\matt_{i,j}$, one has $\LLL(P)(i) \leq \LLL(P)(r+j)$.
We refer to $\matt$ as the \define{underlying matrix} of $P$.
We say that a presentation matrix $P$ as above is a presentation matrix of a persistence module $M : \RR^n \to \vect$ if $M$ is isomorphic to the cokernel of the morphism
$\bigoplus_{r+1 \leq j \leq r+c} F_{\LLL(P)(j)} \to \bigoplus_{1 \leq i \leq r} F_{\LLL(P)(i)}$ that has $\matt$ as its matrix of coefficients.

Let $M,N : \RR^n \to \vect$ be finitely presentable.
Denote by $\PPP_{M,N}$ the set of all pairs $(P_M,P_N)$ consisting of presentation matrices $P_M$ and $P_N$ of $M$ and $N$ respectively, such that $P_M$ and $P_N$ have the same underlying matrix.
Although we do not make use of this in what follows, it is worthwhile observing that $\PPP_{M,N}$ is non-empty if and only if $M$ and $N$ are at finite interleaving distance, by \cite[Theorem~4.4]{lesnick}.
Define a dissimilarity $\dIonehat$ on finitely presentable $n$-parameter persistence modules by
\[
    \dIonehat(M,N) = \inf_{(P_M,P_N) \in \PPP_{M,N}} \|\LLL(P_M) - \LLL(P_N)\|_1 = \inf_{(P_M,P_N) \in \PPP_{M,N}} \sum_{i = 1}^{r+c} \|\LLL(P_M)(i) - \LLL(P_N)(i)\|_1.
\]

As noted in \cite[Example~3.1]{bjerkevik-lesnick}, the dissimilarity $\dIonehat$ does not satisfy the triangle inequality.
This motivates the following definition of the $1$-presentation distance, which is equivalent to the original definition \cite[Definition~3.2]{bjerkevik-lesnick} of Bjerkevik and Lesnick, by \cite[Proposition~3.3(ii)]{bjerkevik-lesnick}.

\begin{definition}
    \label{presentation-distance-def}
The \define{$1$-presentation distance} $\dIone$ is the largest extended pseudodistance on finitely presentable $n$-parameter persistence modules that is bounded above by $\dIonehat$.
\end{definition}

\subsection{Proof of $1$-Wasserstein stability of Hilbert functions}
\label{proof-of-wasserstein-stability-subsection}

Our proof of $1$-Wasserstein stability for $n=2$ makes use of a few technical definitions and results from \cite{bjerkevik-lesnick}.
In order to motivate these definitions and results, we now give an informal outline of the proof.

\begin{proof}[{Outline of the proof of $1$-Wasserstein stability (\cref{wasserstein-stability})} for $n=2$]
    We begin by abstracting as \cref{useful-universality-lp}
    an argument from \cite{bjerkevik-lesnick}
    that provides a sufficient condition for an extended pseudodistance~$d$ to satisfy $d \leq \dIone$.  The condition asks that,
    if we have two presentation matrices $P_M$ and $P_N$ with the same underlying matrix
    $\mu \in \field^{r \times c}$ and such that the labeling functions $\LLL(P_M)$ and $\LLL(P_N)$ induce suitably compatible preorders on the set $\{1, \dots, r+c\}$, then $d(M,N) \leq \|\LLL(P_M) - \LLL(P_N)\|_1$. This compatibility condition is given in \cref{definition:compatibility}. 
    We then use \cref{useful-universality-lp} with the extended pseudodistance $d$
    being $d(M,N) = \widehat{\dwass}(\msd(M),\msd(N))/2$.

    In order to be able to satisfy the hypothesis of \cref{useful-universality-lp}, we show that suitably compatible presentation matrices differ in a grid function, a notion given in \cref{definition:grid-function}.
    This allows us to use a key lemma from~\cite{bjerkevik-lesnick}, recalled as \cref{grid-lemma}.
    This lemma lets us show that, for suitable compatible presentation matrices $P_M$ and $P_N$, the quantity $\|\LLL(P_M) - \LLL(P_N)\|_1$ is an upper bound for the $1$-Wasserstein distance between the Betti numbers in homological degree $2$ of $M$ and $N$.
\end{proof}

We now proceed with the formal definitions and arguments.

\begin{definition}
    \label{definition:compatibility}
    Let $S$ be a set and let $f,g : S \to \RR$ be functions.
    The function $g$ is \define{$f$-compatible} if $f(x) \leq f(y)$ implies $g(x) \leq g(y)$ for all $x,y \in S$.
    If $f,g : S \to \RR^n$, the function $g$ is $f$-compatible if, for every $1 \leq k \leq n$, the function $g_k : S \to \RR$ given by $g_k(s) = g(s)_k$ is $f_k$-compatible.
\end{definition}

The following result is implicit in the proof of \cite[Theorem~1.7(iv)]{bjerkevik-lesnick}; we prove it here for completeness.
Informally, this result reduces the problem of showing stability with respect to the $1$-presentation distance to that of showing stability with respect to changes in the labeling function of a presentation that result in a compatible labeling function, in the sense of \cref{definition:compatibility}.

\begin{lemma}
    \label{useful-universality-lp}
    Let $d$ be an extended pseudodistance on finitely presentable $n$-parameter persistence modules.
    In order to prove $d \leq \dIone$, it is enough to prove the following:
    \begin{itemize}
        \item[$(\ast)$]
    Let $P_M$ and $P_N$ be presentation matrices of
    any finitely presentable $n$-parameter persistence modules $M$ and $N$ with the same underlying matrix 
    and such that $\LLL(P_N)$ is $\LLL(P_M)$-compatible.
    Then $d(M,N) \leq \|\LLL(P_M) - \LLL(P_N)\|_1$.
    \end{itemize}
\end{lemma}
\begin{proof}
    Let $d$ be an extended pseudodistance on finitely presentable $n$-parameter persistence modules satisfying condition~$(\ast)$.
    By definition of the $1$-presentation distance, it is sufficient to prove that $d \leq \dIonehat$, and to check this it is sufficient to prove that, given finitely presentable $n$-parameter persistence modules $M$ and $N$ and presentation matrices $P_M$ and $P_N$ of $M$ and $N$ respectively with the same underlying matrix, we have $d(M,N) \leq \|\LLL(P_M) - \LLL(P_N)\|_1$.
    
    Let $\mu \in \field^{r \times c}$ be the underlying matrix of both $P_M$ and $P_N$.
    Given $t \in [0,1]$, let $P_t$ be the presentation matrix with underlying matrix $\mu$ and such that
    \[
        \LLL(P_t) = (1-t) \cdot \LLL(P_M) + t \cdot \LLL(P_N).
    \]
    Consider the set of
    points $\{t_1, \dots, t_w\} \subseteq (0,1)$ consisting of the $t \in (0,1)$ such that there exists $i,j \in \{1, \dots, r+c\}$ and $t' \in [0,1]$ with
    \[
        \LLL(P_t)_i = \LLL(P_t)_j \text{ and } \LLL(P_{t'})_i \neq \LLL(P_{t'})_j,
    \]
    that is, the set of points where the order of the labels change as $t$ varies from $0$ to $1$.
    Define, also, $t_0 = 0$ and $t_{w+1} = 1$.
    It is easily seen that, for $i \in \{0, \dots, w\}$ and $s \in (t_i, t_{i+1})$, both $\LLL(P_{t_i})$ and $\LLL(P_{t_{i+1}})$ are $\LLL(P_s)$-compatible.
    Finally, for $i \in \{0, \dots, w+1\}$, let $M_i$ be the module that the presentation matrix $P_{t_i}$ is presenting, so that $M \cong M_0$ and $N \cong M_{t_{w+1}}$.

    To conclude, note that
    \begin{align*}
        d(M,N) &\leq d(M_{t_0},M_{t_1}) + \dots + d(M_{t_w},M_{t_{w+1}})\\
               &\leq \|\LLL(P_{t_0}) - \LLL(P_{t_1})\|_1 + \dots + \|\LLL(P_{t_w}) - \LLL(P_{t_{w+1}})\|_1\\
               &= \|\LLL(P_M) - \LLL(P_N)\|_1,
    \end{align*}
    where in the first inequality we used the triangle inequality for $d$, in the second inequality we used condition~$(\ast)$, and in the last equality we used the fact that the presentations $P_t$ were defined by linear interpolation.
\end{proof}

\begin{definition}
    \label{definition:grid-function}
    Let $n \in \N$.
    A \define{grid function} is a function $\XXX : \RR^n \to \RR^n$ of the form $\XXX(a_1, \dots, a_n) = (\XXX_1(a_1), \dots, \XXX_n(a_n))$ such that, for all $1 \leq k \leq n$, we have that $\XXX_k : \RR \to \RR$ is order preserving, left continuous and satisfies $\lim_{a \to \pm\infty} \XXX_k(a) = \pm \infty$.
\end{definition}

The left continuity assumption in the definition of above ensures that the following construction is well-defined.
Given a grid function $\XXX : \RR^n \to \RR^n$, define $\XXX^{-1}(a_1,\dots,a_n) = (\XXX^{-1}_1(a_1),\dots,\XXX^{-1}_n(a_n))$, where, for $1 \leq k \leq n$, we let
\[
    \XXX^{-1}_k(a) = \max\{t \in \RR : \XXX_k(t) \leq a\}.
\]
Given a grid function $\XXX : \RR^n \to \RR^n$ and a module $M : \RR^n \to \vect$, define $E_\XXX(M) = M \circ \XXX^{-1}$.
Similarly, if $f : M \to N$ is a morphism of persistence modules, one defines $E_\XXX(f) : E_\XXX(M) \to E_\XXX(N)$ using the functoriality of precomposition.

\begin{lemma}
    \label{lemma:extend-grid-function}
    Let $S$ be a
    finite set, and let $f,g : S \to \RR^n$ be functions.
    If $g$ is \define{$f$-compatible}, then there exists a grid function $\XXX : \RR^n \to \RR^n$ such that $g = \XXX \circ f$.
\end{lemma}
\begin{proof}
    We first define $\XXX$ on the image of $f$.
    If $z \in \RR^n$ is in the image of $f$, define $\XXX(z) = g(s)$, where $z = f(s)$ for some $s \in S$.
    This is well-defined since $f(s) = f(s')$ implies $g(s) = g(s')$ by the fact that $g$ is $f$-compatible.
    It also follows from the fact that $g$ is $f$-compatible that this definition of $\XXX$
    on the image of $f$ is order-preserving.
    Since $S$ is finite, the image of $f$ is discrete as a subset of $\RR^n$ seen as a metric space, so we can extend $\XXX$ to a grid function.
\end{proof}

An \define{ordered basis} of a free $n$-parameter persistence module $M$ consists of an ordered list $B = \{b_1, \dots, b_k\}$ of elements of $\RR^n$ such that $M$ is isomorphic to $\bigoplus_{1 \leq i \leq k} F_{b_i}$.
If $\XXX : \RR^n \to \RR^n$ is a grid function and $B = \{b_1, \dots, b_k\}$ is an ordered basis of a module $M$, define $E_\XXX(B) = \{\XXX(b_1), \dots, \XXX(b_k)\}$, which is an ordered basis of $E_\XXX(M)$.
If $B = \{b_1, \dots, b_k\}$ and $B'= \{b'_1, \dots, b'_k\}$ are ordered bases of same length $k$, we let $\|B - B'\|_1 = \sum_{1 \leq i \leq k} \|b_i - b'_i\|_1$.

\begin{lemma}[{\cite[Lemma~5.9~(ii)]{bjerkevik-lesnick}}]
    \label{grid-lemma}
    Suppose given a morphism of finitely generated free two-parameter persistence modules $f : P \to Q$, a grid function $\XXX : \RR^2 \to \RR^2$, and ordered bases $B$ and $C$ of $P$ and $\ker(f)$, respectively.
    Let $B' = E_\XXX(B)$ and let $C' = E_\XXX(C)$.
    Then, $\|C - C'\|_1 \leq \|B - B'\|_1$.
\end{lemma}

We can now prove the main result of this section.
\wassersteinstability*
\begin{proof}
    The fact that we have $\widehat{\dwass}(\msd(M),\msd(N)) = \widehat{\dwass}(\sbetti(M),\sbetti(N))$ follows directly from \cref{properties-signed-wass}.
    This establishes the equality of the statement, so it remains to show the inequality.

    We start by proving the case $n=1$.
    Since $M$ and $N$ are finitely presentable, the decomposition theorem for one-parameter persistence modules \cite{zomorodian-carlsson} implies that $M \cong \bigoplus_{i \in I} \field_{[a_i,b_i)}$
    and $N \cong \bigoplus_{j \in J} \field_{[c_j,d_j)}$, where $\{[a_i,b_i)\}_{i \in I}$ and $\{[c_j,d_j)\}_{j \in J}$ are finite multisets of half-open intervals of $\RR$, with the convention that the right endpoint may be $\infty$, so  $\field_{[a,b)} \coloneqq F_a$ if $b = \infty$ and $\field_{[a,b)} \coloneqq \coker(F_b \to F_a)$ if $b < \infty$.
Note that there exists a minimal projective resolution $0 \to F_b \to F_a \to \field_{[a,b)}$ for any $a < b \in \RR$.
    Thus,
    \[
        \sbetti(M) =
        \left(
            \{a_i\}_{i \in I},
            \{b_i\}_{
                \substack{i \in I
                    \\
                b_i \neq \infty}}
        \right)
        \;\; \text{and} \;\;
        \sbetti(N) =
        \left(
            \{c_j\}_{j \in J},
            \{d_j\}_{
                \substack{j \in J
                    \\
                d_j \neq \infty}}
        \right),
    \]
    so the multisets
    $\sbetti_{\even}(M) \cup \sbetti_{\odd}(N)$
    and
    $\sbetti_{\even}(N) \cup \sbetti_{\odd}(M)$
    are indexed by 
    $I \cup \{j \in J : d_j \neq \infty\}$ and $J \cup \{i \in I : b_i \neq \infty\}$, respectively.
Now, the isometry theorem for the Wasserstein and presentation distance between one-parameter persistence modules \cite[Theorem~1.7(iv)]{bjerkevik-lesnick} implies that, if $\dIone(M,N) < \epsilon < \infty$, then there exist subsets $I' \subseteq I$ and $J' \subseteq J$ and a bijection $f : I' \to J'$ such that
    \[
        \left(\sum_{i \in I'} |a_i - c_{f(i)}| + |b_i - d_{f(i)}|\right)
        +
        \left(\sum_{i \in I \setminus I'} |a_i - b_i|\right)
        +
        \left(\sum_{j \in J \setminus J'} |c_j - d_j|\right)
        < \epsilon,
    \]
    with the convention that, for every $x \in \RR\cup \{\infty\}$ we have $|\infty - x| = 0$ if $x = \infty$ and $|\infty - x| = \infty$ otherwise.
    It is now straightforward to see that the function 
    \[
        I \cup \{j \in J : d_j \neq \infty\} \to J \cup \{i \in I : b_i \neq \infty\}
    \]
    that maps
    $i \in I'$ to $f(i)$,
    $i \in I \setminus I'$ to $i$,
    $j \in J'$ to $f^{-1}(j)$,
    and 
    $j \in J \setminus J'$ to $j$,
    is a well-defined bijection of $1$-cost less than $\epsilon$;
    for this, note that we have $b_i\neq\infty$ whenever $i\in I \setminus I'$, as well as $d_j \neq \infty$ whenever $j \in J \setminus J'$, because otherwise $\epsilon = \infty$.
    It follows that $\widehat{\dwass}(\sbetti(M),\sbetti(N)) < \epsilon$, as required.

    \medskip

    We now prove the case $n=2$.
    By \cref{properties-signed-wass}, the mapping $(M,N) \mapsto \widehat{\dwass}(\msd(M),\msd(N))/2$ defines an extended pseudodistance on finitely presentable two-parameter persistence modules.
    Thus, \cref{useful-universality-lp} implies that, in order to prove the claim, it suffices to prove the following:
    given presentation matrices $P_M$ and $P_N$ for finitely presentable two-parameter persistence modules $M$ and $N$; if
    $P_M$ and $P_N$ have the same underlying matrix and $\LLL(P_N)$ is $\LLL(P_M)$-compatible, then
    \[
        \widehat{\dwass}(\msd(M),\msd(N)) \leq 2\cdot \|\LLL(P_M) - \LLL(P_N)\|_1.
    \]

    First note that, by taking a kernel, the presentation matrix $P_M$ induces a resolution
    $0 \to K_M \to P \to Q \to M$, where $P = \bigoplus_{r+1 \leq j \leq r+c} F_{\LLL(P_M)(j)}$ and $Q = \bigoplus_{1 \leq i \leq r} F_{\LLL(P_M)(i)}$.
    Thus, the module $K_M$ is, by definition, the kernel of the morphism $P \to Q$.
    Since the kernel of a morphism between free, finitely presentable two-parameter persistence modules is itself free (see, e.g.,~\cite[Lemma~5.3]{bjerkevik-lesnick}), the module $K_M$ is free.
    Then
    \[
        \widehat{\dwass}\Big(\;\;\msd(M)\;\;,\;\; \big(\B(Q) \cup \B(K_M)\;,\; \B(P)\big)\;\;\Big) = 0,
    \]
    by \cref{properties-signed-wass}(1.) and the fact that $(Q \oplus K_M,P)$ is a Hilbert decomposition of $\hil(M)$, by the rank-nullity theorem.
    By an analogous construction for
    $P_N$, we have a free resolution $0 \to K_N \to P' \to Q' \to N$, so, by the triangle inequality for $\widehat{\dwass}$, it is enough to show that 
    \[
        \dwass\Big(\;\B(Q) \cup \B(K_M) \cup \B(P')\;\;,\;\; \B(Q') \cup \B(K_N) \cup \B(P)\;\Big) \leq 2 \cdot \|\LLL(P_M) - \LLL(P_N)\|_1.
    \]
    By definition, we have
    \begin{align*}
        \B(Q) &= \LLL(P_M)(\{1,\dots,r\}),\\
        \B(P) &= \LLL(P_M)(\{r+1,\dots,r+c\}),\\
        \B(Q') &= \LLL(P_N)(\{1,\dots,r\}),\\
        \B(P') &= \LLL(P_N)(\{r+1,\dots,r+c\}),
    \end{align*}
    and thus
    \begin{align*}
        \dwass&\Big(\B(Q)\cup \B(P')\;\;,\;\; \B(Q') \cup \B(P)\Big) \leq \dwass\Big(\B(Q), \B(Q')\Big) + \dwass\Big(\B(P'), \B(P)\Big)\\
        &\leq \left(\sum_{1 \leq i \leq r} \|\LLL(P_M)(i) - \LLL(P_N)(i)\|_1\right) +
        \left(\sum_{r+1 \leq i \leq r+c} \|\LLL(P_N)(i) - \LLL(P_M)(i)\|_1\right) \\
        &= \|\LLL(P_M) - \LLL(P_N)\|_1.
    \end{align*}
    To conclude, it is sufficient to show that $\dwass(\B(K_M), \B(K_N)) \leq \|\LLL(P_M) - \LLL(P_N)\|_1$.

    Since $\LLL(P_N)$ is $\LLL(P_M)$-compatible, there exists a grid function $\XXX : \RR^2 \to \RR^2$ with the property that $\LLL(P_N) = \XXX \circ \LLL(P_M)$, by \cref{lemma:extend-grid-function}.
    Restricting $\LLL(P_M)$ to $\{r, \dots, r+c\}$ we get a basis of $P$, which we denote by $B$.
    Let $B' = E_\XXX(B)$,
    which is an ordered basis of $P'$, let $C$ be any basis of $K_M$, and let $C' = E_\XXX(C)$, which is an ordered basis of $K_N$.
    By \cref{grid-lemma} we have $\|C - C'\|_1 \leq \|B - B'\|_1$.
    To conclude the proof, we combine the inequality $\|C - C'\|_1 \leq \|B - B'\|_1$ with the inequalities $\|B - B'\|_1 \leq \|\LLL(P_M) - \LLL(P_N)\|_1$ and $\dwass(\B(K_M), \B(K_N)) \leq \|C - C'\|_1$,
    which we now justify.

    The inequality $\|B - B'\|_1 \leq \|\LLL(P_M) - \LLL(P_N)\|_1$ follows from the definitions of $B$ and $B'$, and the fact that $\LLL(P_N) = \XXX \circ \LLL(P_M)$.
    The inequality $\dwass(\B(K_M), \B(K_N)) \leq \|C - C'\|_1$ follows from the fact that $\|C - C'\|_1$ is the $1$-cost of a particular bijection between $\B(K_M)$ and $\B(K_N)$, namely, the one that maps the element of $\B(K_M)$ corresponding to $c_i \in C = \{c_1, \dots, c_\ell\}$ to the element of $\B(K_N)$ corresponding to $c'_i \in C' = \{c'_1, \dots, c'_\ell\}$.
\end{proof}

\subsection{The signed $1$-Wasserstein distance and the Kantorovich norm}

In this short section, we prove that the signed $1$-Wasserstein distance can be seen as a particular case of the Kantorovich norm.
We start by recalling the definition of the Kantorovich norm.

Let $\nu$ be a finite, signed Radon measure on $\R^n$ of total mass $0$, and let $\nu = \nu_+ - \nu_-$ be its Jordan decomposition \cite[p.~421]{billingsley}.
The \emph{Kantorovich norm} of $\nu$ \cite{kantorovich-rubinstein} (see, e.g., \cite{hanin} for a reference in English) is defined as 
\[
    \|\nu\|_1^\Ksf = \inf\left\{ \int_{\R^n \times \R^n} \|x - y\|_1 \; d \psi(x,y)\; :\; \psi \text{ is a coupling between $\nu^+$ and $\nu^-$ } \right\},
\]
that is, as the usual optimal transport distance (with $p=1$) between its positive and negative parts.
The usage of the norm $\|-\|_1$ on $\R^n$ is well-suited for our purposes, but note that other choices such as $\|-\|_2$ are common.

To a finite, signed barcode $\B = (\B_+,\B_-)$, one can assign the finite signed Radon measure $\nu_\B = \sum_{i \in \B_+} \delta_i - \sum_{j \in \B_-} \delta_j$, where $\delta_i$ represents the Dirac measure corresponding to $i$.
The \define{total mass} of a finite signed barcode $\B$ is equal to the cardinality of $\B_+$ minus that of $\B_-$.
Note that, if $\B$ and $\C$ are finite signed barcodes with the same total mass, then the measure $\nu_\B - \nu_\C$ has total mass $0$.

\begin{proposition}
    \label{proposition:wasserstein-is-kantorovich}
    Let $\B$ and $\C$ be finite signed barcodes with the same total mass.
    Then, 
    \[
        \widehat{\dwass}(\B,\C) = \|\nu_\B - \nu_\C\|_1^\Ksf.
    \]
\end{proposition}
\begin{proof}
    Consider the signed barcode $\A = \overline{(\B_+ \cup \C_-,\C_+ \cup \B_-)}$.
    Note that
    \[
        \widehat{\dwass}(\B,\C)
        = \widehat{\dwass}\big( (\B_+ \cup \C_-,\C_+ \cup \B_-) , (\emptyset,\emptyset) \big)
        = \widehat{\dwass}(\A, (\emptyset,\emptyset))
        = \dwass(\A_+, \A_-),
    \]
    where the first equality follows from the fact that
    $\widehat{\dwass}$ is balanced, and for the second equality we use
    \cref{properties-signed-wass}(3.).  Direct inspection shows also
    that $\nu_\B - \nu_\C = \nu_\A$ and thus $\|\nu_\B - \nu_\C\|_1^\Ksf = \|\nu_\A\|_1^\Ksf$.

    Thus, it is enough to prove that, for every finite signed barcode $\A = (\A_+, \A_-)$, we have
    \[
        \dwass(\A_+, \A_-) = 
    \inf\left\{ \int_{\R^n \times \R^n} \|x - y\|_1 \; d \psi(x,y)\; :\; \psi \text{ is a coupling between $\nu_{A^+}$ and $\nu_{A^-}$ } \right\}.
    \]
    Recall that the left hand side is equal to the infimum over all bijections $\A_+ \to \A_-$ of the cost of that bijection.
    Thus, the equality is a particular instance of the well-known fact that the Wasserstein distance between point measures where all points have the same mass is attained by a coupling which is represented by a permutation matrix, or, in other words, that the computation of such Wasserstein distances reduces to an assignment problem; see, e.g., \cite[Proposition~2.1]{peyre-cuturi}.
\end{proof}

\section{Algorithmic considerations}
\label{algorithm-section}

In this section, we address the computability of the dissimilarities $\widehat{d_B}$ and $\widehat{\dwass}$.
For completeness, we first briefly address the computability of multigraded Betti numbers.

\subsection{Computing multigraded Betti numbers and minimal Hilbert decompositions}

\subsubsection{Multigraded Betti numbers}
\label{sec:compute_Betti}
The multigraded Betti numbers of a finitely generated bigraded $\field[x,y]$-module~$M$ can be computed in polynomial time using the Lesnick--Wright algorithm \cite{lesnick-wright-2}.
The algorithm runs in time $O(|X|^3 + |Y|^3 + |Z|^3)$, where the module is given by $M \cong \ker(g)/\im(f)$ for a pair of morphisms $X \xrightarrow{f} Y \xrightarrow{g} Z$ between free modules satisfying $g \circ f = 0$, and where $|-|$ denotes the rank (i.e., the number of generators). Here are two typical practical scenarios in which this algorithm can be applied:
\begin{itemize}
\item When a bifiltered simplicial complex (i.e., a simplicial complex endowed with two real-valued filtering functions) is given as input,
and~$M$ is defined as its persistent homology in degree~$r$, in which case $X, Y, Z$ are the free modules generated by the $(r+1)$-, $r$- and $(r-1)$-simplices of the complex respectively. The running time of the algorithm is then in~$O(m^3)$ where $m$ is the total number of those simplices.
\item When a (possibly non-minimal) free presentation $X\twoheadrightarrow Y$ of~$M$ is given as input, in which case we have $Z=0$ and the algorithm runs in~$O(m^3)$ time, where $m$ is the total number of generators in the presentation.
\end{itemize}

Let us point out that optimizations to the Lesnick--Wright algorithm were introduced in \cite{kerber-rolle}, which significantly improve the performance of the algorithm in practice while not changing its theoretical worst-case complexity. For $n> 2$, the development of efficient algorithms to compute minimal presentations of multiparameter persistence modules, which in particular can be leveraged to compute multigraded Betti numbers, is an actively explored research direction~\cite{bender-gafvert-lesnick}.

\subsubsection{Minimal Hilbert decompositions}
\label{sec:compute_minHil}

By \cref{remark-same-hilbert-function}, one can reduce the computation of the minimal Hilbert decomposition signed barcode of a finitely presentable $n$-parameter persistence module to that of its multigraded Betti numbers.
We now give an efficient algorithm to compute this minimal Hilbert decomposition directly from the module's Hilbert function.
In order to do this, we recall the basics from the theory of M\"obius inversion.
For details, see, e.g., \cite[Chapter~3.7]{stanley}.

\begin{proposition}[M\"obius inversion formula]
    \label{proposition:minversion}
    Let $\PP$ be a finite poset.
    There exists a unique function $\mu : \{(s,t)\in \PP \times \PP : s \leq t\} \to \Z$ with the following property.
    For every pair of functions $f,g : \PP \to \Z$ we have that
    \[
        g(t) = \sum_{s \leq t} f(s) \text{ for all } t \in P
    \]
    if and only if
    \[
        f(t) = \sum_{s \leq t} g(s) \mu(s,t) \text{ for all } t \in P.
    \]
\end{proposition}

The function $\mu$ is called the \define{M\"obius function} of the poset $\PP$ and is denoted by $\mu_\PP$ when the poset may not be clear from the context.
Whenever we have functions $f$ and $g$ as in \cref{proposition:minversion}, we say that $f$ is the M\"obius inverse of $g$.

Next, we show that the M\"obius function of finite grids, i.e., finite products of finite linear posets, have a particularly simple form.

Fix $\ell \geq 1 \in \Z$ and let $[\ell] = \{0, 1, \dots, \ell-1\}$ with its usual order.
Given $n \geq 1 \in \Z$, let $E = \{e_1, \dots, e_n\}$ denote the canonical basis of $\R^n$.
Given a subset $S \subseteq E$, let $\Sigma S = \sum_{e \in S} e \in \R^n$, with the convention that $\Sigma S = (0, \dots, 0) \in \R^n$ if $S = \emptyset$.

\begin{lemma}
    \label{lemma:minversion-grid}
    Consider the poset $[\ell]^n$ and let $s \leq t \in [\ell]^n$.
    Then $\mu(s,t) \neq 0$ if and only if there exists $S \subseteq E$ such that $s = t - \Sigma S$.
    Moreover, for every $t \in [\ell]^n$ and $S \subseteq E$ such that $t-\Sigma S \in [\ell]^n$, we have
    \[
        \mu(t-\Sigma S,t) = (-1)^{|S|}.
    \]
\end{lemma}
\begin{proof}
    The case $n=1$ is a special case of the formula for the M\"obius function for linear orders \cite[Example~3.8.1]{stanley}.
    The general case follows directly from the case $n=1$ and the fact that the M\"obius function is multiplicative \cite[Proposition~3.8.2]{stanley}, in the sense that $\mu_{\PP \times \QQ}((s,s'),(t,t')) = \mu_{\PP}(s,s') \mu_{\QQ}(t,t')$ for any pair finite posets $\PP$ and $\QQ$ and all $s \leq s' \in \PP$ and $t \leq t' \in \QQ$.
\end{proof}

In order to compute the minimal Hilbert decomposition signed barcode of a finitely presentable module, we will assume that the given module is an extension of a module over a finite grid, in the following sense.
Given $M : [\ell]^n \to \vect$, define the extension persistence module $\widehat{M} : \RR^n \to \vect$ pointwise by
\[
    \widehat{M}(r) =
    \begin{cases}
        0& \text{if any of the coordinates of $r$ is negative }\\
        M(s) & \text{otherwise, where $s = \max \{t \in [\ell]^n : t \leq r\}$}.
    \end{cases}
\]
Define the structure morphisms of $\widehat{M}$ by letting the structure morphisms $\widehat{M}(r) \to \widehat{M}(r')$ be zero if any of the coordinates of $r$ is negative and, otherwise, by letting it be equal to the structure morphism $M(s) \to M(s')$ where $s = \max \{t \in [\ell]^n : t \leq r\}$ and $s' = \max \{t' \in [\ell]^n : t' \leq r'\}$.

\begin{proposition}
    \label{proposition:hilbert-decomposition-from-minversion}
    Let $\ell \geq 1 \in \Z$ and let $M : [\ell]^n \to \vect$.
    Define a function $\gamma : [\ell]^n \to \Z$ by letting
    \begin{align*}
        \gamma(i) 
            &= \sum_{\substack{S \subseteq E\\ i-\Sigma S\,\, \in\,\, [\ell]^n}} (-1)^{|S|}\; \dim(M(i-\Sigma S)).
    \end{align*}
    The following is a minimal Hilbert decomposition of $\widehat{M}$:
    \[
        \left(\; \bigoplus_{\substack{i \in [\ell]^n\\ \gamma(i) > 0}} F_i^{\gamma(i)}\;,\;
               \bigoplus_{\substack{i \in [\ell]^n\\ \gamma(i) < 0}} F_i^{-\gamma(i)} \;\right).
    \]
\end{proposition}
\begin{proof}
    Let $(P,Q)$ be the pair that is claimed to be a minimal Hilbert decomposition of $\widehat{M}$.
    The fact that $(P,Q)$ is minimal is clear, since $\gamma(r)$ is either zero, strictly positive, or strictly negative, so we only need to show that it is a Hilbert decomposition of $\widehat{M}$.
    Since $\widehat{M}$ is an extension of a module defined over the finite grid $[\ell]^n$, it is enough to prove that $\hil(P)(r) - \hil(Q)(r) = \hil(M)(r)$ for every $r \in [\ell]^n$.
    
    First, note that, if $r \in [\ell]^n$, then
    \[
        \hil(P)(r) - \hil(Q)(r) = \hil\left(\bigoplus_{\substack{i \in [\ell]^n\\ \gamma(i) > 0}} F_i^{\gamma(i)}\right)(r) - 
            \hil\left(\bigoplus_{\substack{i \in [\ell]^n\\ \gamma(i) < 0}} F_i^{-\gamma(i)}\right)(r) =
        \sum_{\substack{i \in [\ell]^n\\ \gamma(i) > 0\\ i \leq r}} \gamma(i) + \sum_{\substack{i \in [\ell]^n\\ \gamma(i) < 0\\ i \leq r}} \gamma(i) = \sum_{\substack{i \in [\ell]^n\\i \leq r}} \gamma(i),
    \]
    where in the second equality we used the fact that $\hil(F_i)(j) = 1$ if $i \leq j$ and $0$ otherwise.
    Thus, by \cref{proposition:minversion}, we have that $\gamma$ is the M\"obius inverse of the function $[\ell]^n \to \Z$ given by mapping $r \in [\ell]^n$ to $\hil(P)(r) - \hil(Q)(r)$.
    But by \cref{lemma:minversion-grid}, we have that $\gamma$ is also the M\"obius inverse of the function $[\ell]^n \to \Z$ given by mapping $r$ to $\dim(M(r))$.
    Thus, $\hil(P)(r) - \hil(Q)(r) = \hil(M)(r)$ for every $r \in [\ell]^n$.
\end{proof}

\begin{remark}
    \label{remark:computation-convolution}
Given the value of the Hilbert function $\eta : [\ell]^n \to \Z$ of the extension of a module on the finite grid $[\ell]^n$, one can use \cref{proposition:hilbert-decomposition-from-minversion} to compute the minimal Hilbert decomposition of the module by performing $|\mathrm{Parts}(E)| = 2^n$ additions of vectors of size $\ell^n$, and thus in $O((2\ell)^n)$ time,
where $\mathrm{Parts}(E)$ denotes the set of subsets of $E$, and $E$ denotes the canonical basis of $\R^n$.

Using basic notions from multidimensional signal processing \cite{dudgeon-mersereau}, this time complexity can be improved to $O(n \cdot \ell^n)$, as follows.
First, define $h : \Z^n \to \Z$ by $h(x_1, \dots, x_n) = h'(x_1) \cdots h'(x_n)$, where $h' : \Z \to \Z$ is defined by $h'(0) = -1$, $h'(1) = 1$, and $0$ otherwise.
Secondly, extend $\eta : [\ell]^n \to \Z$ by zeros to a function $\eta' : \Z^n \to \Z$.
Then, by definition of $h$ and by the formula in \cref{proposition:hilbert-decomposition-from-minversion}, the function $\gamma : [\ell]^n \to \Z$ is obtained by performing the $n$-dimensional convolution of $h$ and $\eta$,
then by restricting to $[\ell]^n$. 
Note that $h$ is \emph{separable}, meaning that it is the product of $n$ functions of one variable.
This implies that $\gamma$ can be computed by performing $n \cdot \ell^{n-1}$ one-dimensional convolutions \cite[Section~1.2.6]{dudgeon-mersereau}.
Since we are only interested in the restriction of $h \ast \eta$ to $[\ell]^n$, and the support of $h'$ is only $\{0,1\}$, each one of the one-dimensional convolutions can be restricted to a vector of size $\ell+1$, so each one-dimensional convolution takes time $O(\ell)$ and thus $\gamma$ can be computed in time $O(n \cdot \ell^n)$.
\end{remark}

\subsection{Computing dissimilarities between signed barcodes}
\label{sec:compute_dists}

Consider finite $n$-dimensional signed barcodes $\B=(\B_+, \B_-)$ and $\C=(\C_+, \C_-)$.
By definition, $\widehat{d_B}(\B,\C) = d_B(\B_+ \cup \C_-, \C_+ \cup \B_-)$ and $\widehat{\dwass}(\B,\C) = \dwass(\B_+ \cup \C_-, \C_+ \cup \B_-)$, so we only consider the problem of computing $d_B$ and $\dwass$ between unsigned barcodes.

In the rest of this section, we let $\B$ and $\C$ be finite $n$-dimensional unsigned barcodes, and we let $b = |\B|$ and $c = |\C|$.
Note that, if $b \neq c$, then $d_B(\B,\C) = \dwass(\B,\C) = \infty$, so we assume $b=c=K$.

\subsubsection*{Bottleneck distance}
Our proposed approach for computing $d_B$ is analogous to the usual approach to computing the bottleneck distance between usual, one-parameter persistence barcodes.
In fact, our situation is simpler, as we only need to consider perfect matchings.

First, note that $d_B(\B,\C)$ can only take one of the $K^2$ values $d_{ij} = \|i - j\|_\infty$ for $i \in \B$ and $j \in \C$.
One can do a binary search on the sorted sequence of lengths $d_{ij}$ as follows.
Given a length $d_{ij}$, construct the bipartite graph $G_{ij}$ on $(\B,\C)$ that has an edge between $k\in \B$ and $l \in \C$ if and only if $\|k - l\|_\infty \leq d_{ij}$.
Run the Hopcroft--Karp maximum cardinality matching algorithm on $G_{ij}$.
If the matching covers all vertices, then $d_B(\B,\C) \leq d_{ij}$; otherwise $d_B(\B,\C) > d_{ij}$.
The Hopcroft--Karp algorithm has runtime in $O(K^{2.5})$, and thus the runtime of the algorithm outlined above is $O(K^2 n + K^{2.5}\log(K))$, where the first term accounts for the computation of the distances $d_{ij}$, and the second term accounts for the sorting of the distances and the binary search.

Optimizations using geometric data structures are possible.
For instance, algorithms of Efrat, Atai, and Katz can be used to compute $d_B(\B,\C)$ in time $O(K^{1.5} \log(K))$ if $n=2$ (see \cite[Theorem~5.10]{efrat-atai-katz}), and in time $O(K^{1.5} \log^n(K))$ for general $n \geq 3$ (see \cite[Theorem~6.5]{efrat-atai-katz}).
Here, as in the rest of the paper, $n$ denotes the number of parameters, so that $\B$ and $\C$ are $n$-dimensional barcodes.

\subsubsection*{$1$-Wasserstein distance}
Arguably, the simplest way to compute $\dwass(\B,\C)$ is to use the Hungarian method~\cite{kuhn,munkres}.
Specifically, one constructs the full bipartite graph on $(\B,\C)$, and weighs each edge $(i,j)$ for $i \in \B$ and $j \in \C$ by $\|i - j\|_1$, and then uses the Hungarian method to find a matching of minimal total cost.
There are several more efficient variants of the Hungarian method, such as the Jonker--Volgenant algorithm, which runs in time $O(K^3)$.

In this case, too, there is space for improvement.
As far as practical performance is concerned, one can use the auction algorithm of Bertsekas \cite{bertsekas} for finding an exact or approximate solution to the minimum-weight perfect matching problem, as has been done in \cite{kerber-morozov-nigmetov} to compare usual, one-parameter persistence barcodes.
Although it is not guaranteed that this approach will perform better than, e.g., the Jonker--Volgenant algorithm,
for computing $\dwass$ between $n$-dimensional barcodes, experience from the one-parameter case suggests so.

\section{Consequences}
\label{consequences-section}

\subsection{Stability of Hilbert functions}
    \label{hilbert-stability-section}
Let $n \in \N$, and let $\persmod$ denote the set of isomorphism classes of finitely presentable $n$-parameter persistence modules.
Define $\hils = \{\hil(M) : M \in \persmod \}$, the set of all Hilbert functions of finitely presentable $n$-parameter persistence modules.
Given $\eta \in \hils$, let $M \in \persmod$ be such that $\eta = \hil(M)$.
Note that the signed barcode $\msd(M)$ only depends on $\eta$ and not on the specific choice of $M$.
In particular, the following extended distance on $\hils$ is well-defined.

\begin{definition}
    \label{distance-on-hilbert-functions}
Let $\dwasshils$ be the extended distance on $\hils$ defined on $\eta,\theta \in \hils$ as
\[
    \dwasshils(\eta,\theta) = \widehat{\dwass}(\msd(M),\msd(N)),
\]
where $M$ and $N$ are any two persistence modules such that $\eta = \hil(M)$ and $\theta = \hil(N)$.
\end{definition}

Note that the above is indeed an extended distance and not just an extended pseudodistance, since, by \cref{corollary:same-hilbert-function}, $\widehat{\dwass}(\msd(M),\msd(N)) = 0$ implies $\hil(M) = \hil(N)$.

\cref{wasserstein-stability} can then be interpreted as a stability result for Hilbert functions, which we state as the following corollary.

\begin{corollary}
    \label{hilbert-stability-corollary}
    Let $n \in \{1,2\}$.
    For finitely presentable $n$-parameter persistence modules $M$ and $N$ we have
    \[
    \dwasshils(\hil(M),\hil(N)) \leq n\cdot \dIone(M,N).
    \]
\end{corollary}

\subsection{Stability of Betti numbers of sublevel set persistence}
Given a CW-complex $K$, a monotonic function $f : \cells(K) \to \RR^n$, and $i \in \N$, consider the finitely presentable $n$-parameter persistence module $H_i(f) := H_i(S(f); \field)$, where, for $x \in \RR^n$, we let $S(f)(x) = \{\sigma \in \cells(K) : f(\sigma) \leq x\} \subseteq K$.
Define $\|f\|_\infty = \max_{\sigma \in \cells(K)} \|f(\sigma)\|_\infty$ and $\|f\|_1 = \sum_{\sigma \in \cells(K)} \|f(\sigma)\|_1$.

\begin{corollary}
    \label{stability-betti-sublevel}
    Let $K$ be a finite CW-complex and let $f,g : \cells(K) \to \RR^n$ be monotonic.
    Then
    \[
        \max_{i \in  \N} \widehat{d_B}\big(\sbetti(H_i(f)), \sbetti(H_i(g))\big) \leq (n^2-1) \,\|f - g\|_\infty.
    \]
    If $n\in \{1,2\}$, we also have
    \[
        \sum_{i \in \N} \widehat{\dwass}\big(\msd(H_i(f)), \msd(H_i(g))\big) \leq n^2 \cdot \|f - g\|_1.
    \]
\end{corollary}
\begin{proof}
    This is a direct consequence of \cref{main-theorem}, \cref{wasserstein-stability}, and the $p$-stability of one- and two-parameter sublevel set persistent homology proven by Bjerkevik and Lesnick \cite[Theorem~1.9(i)]{bjerkevik-lesnick}.
\end{proof}

\subsection{Stability of invariants from exact structures}
As mentioned in the introduction, our work is inspired by the work of Botnan et al.~\cite{botnan-oppermann-oudot} on signed decompositions of the rank invariant.
Indeed, our definition of Hilbert decomposition is analogous to their definition of \emph{rank decomposition}.
The difference is that, while we decompose the Hilbert function of a persistence module as a $\Z$-linear combination of Hilbert functions of free modules, they decompose the rank invariant of a persistence module as a $\Z$-linear combination of rank invariants of rectangle modules.

The analogy goes further.
While we rely on free (hence projective) resolutions, and on the usual notion of exactness, they show that rank decompositions are related to a different exact structure on the category of (multiparameter) persistence modules, the so-called \emph{rank-exact} structure.
There is an analogue of multigraded Betti numbers for this other exact structure, and in on-going work we show that using this notion one can prove an analogue of \cref{main-theorem}.
We use a general result for signed barcodes coming from exact structures, which can be proven following our proof of \cref{main-theorem}, by noticing that Schanuel's lemma holds in any exact structure.
We give this result next.

For our purposes, an \define{exact structure} on an abelian category $\mathbf{C}$ consists of a collection $\EEE$ of short exact sequences of $\mathbf{C}$ that contains the split exact sequences, and that is closed under isomorphisms, pullbacks, and pushouts.
For an introduction to exact structures, see, e.g., \cite{buhler}.
An object $P \in \mathbf{C}$ is \define{$\EEE$-projective} if the functor $\Hom_{\mathbf{C}}(P,-) : \mathbf{C} \to \ab$ maps short exact sequences in $\EEE$ to short exact sequences of $\ab$.
An \define{$\EEE$-projective} resolution of an object $X \in \mathbf{C}$ consists of an exact sequence $\cdots \to P_k \xrightarrow{d_k} P_{k-1} \to \cdots \to P_0 \to X$ such that $P_k$ is $\EEE$-projective for all $k \in \N$, and such that $0 \to \ker(d_k) \to P_k \to \coker(d_{k+1}) \to 0$ is in $\EEE$ for all $k \in \N$.
An $\EEE$-projective resolution of an object $X \in \mathbf{C}$ is \define{minimal} if it is a retract of any other $\EEE$-projective resolution of $X$.
If every object of $\mathbf{C}$ admits a minimal resolution, then the \define{global dimension} of $\EEE$ is the supremum of the lengths of all the minimal resolutions; note that this number can be infinite.

Let $\EEE$ be an exact structure on the category of finitely presentable $n$-parameter persistence modules and denote by $\PPP$ the class of projectives of $\EEE$.
Assume the following:
\begin{itemize}
    \item[(a)] For every $\epsilon \geq 0$, the $\epsilon$-shift of any $\EEE$-exact sequence is $\EEE$-exact.
    \item[(b)] The elements of $\PPP$ are interval decomposable.
    \item[(c)] All finitely presentable $n$-parameter persistence modules admit a finite minimal $\EEE$-projective resolution.
\end{itemize}

Conditions (b) and (c) imply that to every finitely presentable $M : \RR^n \to \vect$ we can associate a signed interval barcode (i.e., a pair of multisets of intervals of the poset $\RR^n$), which is unique up to isomorphism of multisets, by letting:
\[
    \displaystyle (\pi_{2\N}(M), \pi_{2\N+1}(M)) := \left(\bigcup_{j \text{ even}} \B(P_j), \bigcup_{j \text{ odd}} \B(P_j)\right),
\]
where $P_\bullet \to M$ is a minimal $\EEE$-projective resolution of $M$.
Noticing that \cref{persistent-schanuel} applies to any exact structure, as long as condition (a) is satisfied, one can prove the following result, by following the proof of \cref{main-theorem}.

\begin{theorem}
    \label{relative-stability-theorem}
Let $\EEE$ be an exact structure on the category of finitely presentable $n$-parameter persistence modules satisfying conditions (a), (b), and (c), above.
Assume that $\EEE$ has finite global dimension bounded above by $d \in \N$, and that $\EEE$-projective modules are stable, in the sense that there exists $c \geq 1 \in \R$ such that $d_B(\B(P),\B(Q)) \leq c \cdot d_I(P,Q)$ for all $P,Q \in \PPP$.
Then, for all finitely presentable modules $M,N : \RR^n \to \vect$, we have
\[
      d_B\big(\;\pi_{2\N}(M) \cup \pi_{2\N + 1}(N)\;,\;
          \pi_{2\N}(N) \cup \pi_{2\N + 1}(M)\;\big)
            \leq c \cdot (d+1) \cdot  d_I(M,N).
\]
\end{theorem}

Recent work of Blanchette, Br\"{u}stle, and Hanson \cite{blanchette} considers the general approach of defining invariants of persistence modules indexed by finite posets using exact structures.
In particular, they introduce a family of exact structures---which includes the rank-exact structure for finite posets---and show that these structures have finite global dimension.
However, since their argument depends on the cardinality of the indexing poset, it is not clear to us that their global dimension result can be used to provide a finite upper bound for the global dimension of the rank-exact structure on finitely presentable $n$-parameter persistence modules, since $\RR^n$ is not a finite poset.

A last point related to \cite{botnan-oppermann-oudot} is that an example similar to \cref{instability-minimal-decomposition-example} shows that minimal rank decompositions in the sense of \cite{botnan-oppermann-oudot} are not stable in the bottleneck dissimilarity.

\section{Discussion}
\label{discussion-section}

\subsection*{Our lower bound versus the matching distance}
By \cref{main-theorem}, $\widehat{d_B}/(n^2 - 1)$ is a lower bound on the interleaving distance.
As such, it should be compared to other existing lower bounds, and, in particular, to the most popular one of them: the matching distance $\dmatch$ \cite{landi}.

First, as explained in \cref{algorithm-section}, there are efficient algorithms to compute our lower bound exactly in any dimension~$n$.
By contrast, there is currently no known algorithm to compute $\dmatch$ exactly (or within a provable error of $\epsilon$) in dimension $n > 2$.

Now, in the two-parameter case, given (possibly non-minimal) presentations of two modules $M$ and~$N$, the current asymptotically fastest algorithm to compute exactly the matching distance between $M$ and $N$ has expected running time in $O(m^5 \log^3(m))$~\cite{bjerkevik-kerber}, where~$m$ is the total number of generators in the presentations of $M$ and $N$. From the same input, we can use the algorithm of \cref{sec:compute_Betti} to compute the bigraded Betti numbers of $M$ and~$N$ in~$O(m^3)$ time, then the algorithm of \cref{sec:compute_dists} to compute the bottleneck dissimilarity between their associated signed barcodes in~$O(K^{1.5}\log K)$ time, where $K = |\betti_{2\N}(M)| + |\betti_{2\N+1}(N)|$.
Note that $K \in O(m)$.
In order to see this, note that $|\betti_{0}(M)| \in O(m)$ and $|\betti_{1}(M)| \in O(m)$, since a minimal free presentation is a direct summand of any free presentation.
Moreover, we also have $|\betti_{2}(M)| \leq |\betti_{1}(M)| \in O(m)$, since, if $0 \to P_2 \to P_1 \to P_0 \to M$ is a minimal free resolution of $M$, then the morphism $P_2 \to P_1$ is necessarily injective, and thus the rank of $P_2$ is at most that of $P_1$.

All in all, our approach is asymptotically two orders of magnitude (in~$m$) faster than the state of the art. The gap increases further in scenarios where many more distance computations are performed than the number of persistence modules involved (consider for instance the case of metric-based machine learning methods requiring a super-linear number of distance computations): in such scenarios, once the multigraded Betti numbers of the modules have been computed, each bottleneck dissimilarity computation takes only $O(m^{1.5}\log m)$ time.

Resorting to approximate matching distance computations may reduce this gap to some extent but not entirely. The current fastest algorithm to approximate the matching distance between two-parameter persistence modules within an additive error of $\epsilon$ runs in $O(m^3 \epsilon^{-2})$ time when the modules are given as the homology of bifiltered simplicial complexes with at most $m$ simplices \cite{kerber-nigmetov}.
By comparison, from the same input, our approach from \cref{algorithm-section} runs in $O(m^3)$ time.

In terms of discriminative power, it is not yet clear how our bound compares to the matching distance or its approximate version.
Examples like \cref{possible-issue-example} show that there is no constant $c > 0$ such that $\dmatch(M,N) \leq c\cdot \widehat{d_B}(\sbetti(M), \sbetti(N))$ for all finitely presentable $M, N : \RR^n \to \vect$.
However, it is not clear to us whether there exists a constant $c > 0$ such that $\widehat{d_B}(\sbetti(M), \sbetti(N)) \leq c\cdot \dmatch(M,N)$ for all finitely presentable $M, N : \RR^n \to \vect$.
The practical discriminative powers of $\widehat{d_B}$ and $\dmatch$ will be assessed in future work.

For completeness, let us compare with the complexity of computing $\widehat{\dwass}(\sbetti(M), \sbetti(N))$ in the 2-parameter setting.
Using the approach of \cref{algorithm-section}, we can compute this distance in $O(m^3)$ time, either from presentations of $M$ and $N$ with at most $m$ generators in total, or from two bifiltered complexes of size at most $m$ whose homologies are isomorphic to $M$ and $N$ respectively.

\subsection*{Minimal Hilbert decompositions in the context of machine learning}
When used as descriptors in machine learning, persistence barcodes of one-parameter persistence modules are often not compared directly using the ambient bottleneck or Wasserstein distance, but instead, using some vector norm after vectorizing the barcodes in a suitable way---see, e.g., \cite{adams-et-al,bubenik,carriere2020perslay}.
It is usually the case that these vectorizations are stable operations when the space of barcodes is originally equipped with a Wasserstein distance.
This makes \cref{wasserstein-stability} particularly relevant for machine learning, as it suggests that one should look for vectorizations of signed point measures in~$\R^n$ that are stable with respect to the signed $1$-Wasserstein distance.
This is the subject of current investigation.

As further motivation, we mention that \cref{proposition:hilbert-decomposition-from-minversion} gives an efficient and straightforward way of computing the minimal Hilbert decomposition of $n$-parameter persistence modules for arbitrary $n$, without requiring the computation of presentations or of multigraded Betti numbers.
Indeed, when restricted to an $\ell \times \dots \times \ell$ subgrid of $\RR^n$ for some $\ell \in \N_{>0}$, the Hilbert function of the homology in any degree of an $n$-parameter filtration of a simplicial complex with $m$ simplices can be computed in time $O(m^3\, \ell^{n-1})$.
In order to do this, one notes that the Hilbert function of a homology module of a filtration over a grid $[\ell]^n$ can be computed using $\ell^{n-1}$ runs of an ordinary one-parameter persistence algorithm, by restricting the filtration to individual lines parallel to a fixed coordinate axis.
Such algorithms have a theoretical worst-case complexity of $O(m^3)$ but, in practice, efficient implementations such as \cite{bauer} have near-linear running time.
Then, given the Hilbert function, one can apply \cref{remark:computation-convolution} to compute its minimal decomposition over the grid $[\ell]^n$ in $O(n \ell^n)$ time.
This approach has the added advantage that it does not require the filtration to be $1$-critical and thus works for, e.g., the degree-Rips bifiltration \cite{blumberg-lesnick,lesnick-wright}.

\subsection*{Relationship between \cref{main-theorem} and \cref{wasserstein-stability}}
Let $p \in [1,\infty)$.
Given $h : \B \to \C$ a bijection of finite multisets of elements of $\RR^n$, define the $p$-cost of $h$ by $\left(\sum_{i \in \B} \|i - h(i)\|_{p}^p\right)^{1/p}$.
For $p = \infty$ the $p$-cost is defined to be $\max_{i \in \B} \|i - h(i)\|_{\infty}$.
For $p \in [1,\infty]$, define an extended pseudodistance on finite unsigned barcodes by
\[
    \dwassp{p}(\B,\C) = \inf\big\{\,\epsilon \geq 0 \,: \,\exists\text{ bijection }\, h : \B \to \C \text{ with $p$-cost} \leq \epsilon\,\big\} \in \R_{\geq 0}\cup\{\infty\}.
\]
Finally, define the \define{signed $p$-Wasserstein dissimilarity} on finite signed barcodes as $\widehat{\dwassp{p}}$.
Note that, for $p = 1$, this recovers the signed $1$-Wasserstein distance, and for $p = \infty$, this recovers the bottleneck dissimilarity.
It is interesting to note that $\widehat{\dwassp{p}}$ satisfies the triangle inequality if and only if $p=1$.

\cref{wasserstein-stability} says that, for finitely presentable two-parameter persistence modules $M$ and $N$, we have
\begin{equation}
    \label{wasserstein-stability-betti}
    \widehat{\dwass}(\sbetti(M),\sbetti(N)) \leq 2\cdot \dIone(M,N).
\end{equation}
Similarly, when $n=2$, \cref{main-theorem} says that, under the same assumptions, we have $\widehat{d_B}(\sbetti(M),\sbetti(N)) \leq 3\cdot d_I(M,N)$, or equivalently,
\begin{equation}
    \label{stability-betti-as-wasserstein}
    \widehat{\dwassp{\infty}}(\sbetti(M),\sbetti(N)) \leq 3\cdot d^\infty_I(M,N).
\end{equation}
Recall that $d^p_I$ stands for the $p$-presentation distance of Bjerkevik and Lesnick \cite{bjerkevik-lesnick}.

Although \cref{wasserstein-stability-betti} and \cref{stability-betti-as-wasserstein} look very similar, it is interesting to note how different the proofs of \cref{main-theorem} and \cref{wasserstein-stability} are.
This motivates the question of whether the two results are the $p=1$ and $p=\infty$ instances of a general $p$-Wasserstein stability result for multigraded Betti numbers.

Specifically, we ask the following.

\begin{question}
    Given $p \in [1,\infty]$ and finitely presentable $M,N : \RR^n \to \vect$, does there exist a constant $c$ depending only on $p$ and $n$ such that
\[
    \widehat{\dwassp{p}}(\sbetti(M),\sbetti(N)) \leq c\cdot d^p_I(M,N)\;?
\]
\end{question}

We remark that, as we show next, it is possible to give a constant $c$ that depends on $p$, on $n$, and on the modules $M$ and $N$.
Assume that $|\betti_{2\N}(M)| + |\betti_{2\N+1}(N)| = |\betti_{2\N}(N)| + |\betti_{2\N+1}(M)|$ since, otherwise, $\widehat{\dwassp{p}}(\sbetti(M),\sbetti(N))$ cannot be finite.
Assume, moreover, that $n\geq 2$; a bound for $n=1$ is obtained in an analogous way.
Then, letting $K = |\betti_{2\N}(M)| + |\betti_{2\N+1}(N)|$ and using the standard Lipschitz equivalence of norms on finite-dimensional real vector spaces, we have
\begin{align*}
    \widehat{\dwassp{p}}(\sbetti(M),\sbetti(N)) &\leq K^{1/p}\; \widehat{d_B}(\sbetti(M),\sbetti(N))\\ 
     &\leq K^{1/p}\; (n^2-1) \; d_I(M,N) \leq K^{1/p}\; (n^2-1)\; d^p_I(M,N).
\end{align*}

In particular, restricting to finitely presentable $n$-parameter persistence modules that have at most $K$ summands in their minimal resolution, we have $\widehat{\dwassp{p}}(\sbetti(M),\sbetti(N))\leq (2K)^{1/p}\; (n^2-1)\; d^p_I(M,N)$.

\subsection*{Tightness of \cref{main-theorem}}
Hilbert's global dimension
bound is tight: $\gldim(\field[x_1,\dots,x_n]) = n$. Meanwhile,
Bjerkevik's bottleneck stability bound for free modules is tight at
least in the cases $n \in \{2,4\}$ \cite[Example~5.2]{bjerkevik}.
Nevertheless, we do not know whether our
bound in \cref{main-theorem} is tight, even when $n \in \{2,4\}$.
Addressing the tightness of our bound requires understanding exactly
what kinds of matchings can arise from applying Bjerkevik's result on the
interleaving given by \cref{stability-resolutions}, which is the subject of future work. Note that the bound
$\widehat{d_B}\big(\sbetti(M),\, \sbetti(N)\big)\, \leq\, 2\cdot
d_I(M,N)$ obtained in the case $n=1$ is tight: for instance,
$\widehat{d_B}\big(\sbetti(\field_{[0,2)}),\, \sbetti(0)\big)\, =\, 2$
  while the two modules are $1$-interleaved.

  Let us also point out that $\widehat{d_B}\big(\sbetti(M),\, \sbetti(N)\big)$ can be arbitrarily small compared to $d_I(M,N)$, even when $n=1$; indeed, two non-isomorphic finitely presentable modules with the same Hilbert function satisfy $d_I(M,N) \neq 0$ and $\widehat{d_B}\big(\sbetti(M),\, \sbetti(N)\big) = 0$, by \cref{corollary:same-hilbert-function}. This is expected of any distance or dissimilarity function on multigraded Betti numbers, since the latter  forget about the differentials in the resolutions---hence about the pairing defining the intervals in the persistence barcode in the case $n=1$.  For a concrete example, take for instance $M=\field_{[0,2)} \oplus \field_{[1,3)} : \RR \to \vect$ and $N=\field_{[0,3)} \oplus \field_{[1,2)} : \RR \to \vect$, for which we have $\sbetti(M) = \sbetti(N)$ while $M\not\cong N$, so  $\widehat{d_B}\big(\sbetti(M),\, \sbetti(N)\big)=0$ while $d_I(M,N) = d_B\big(\B(M),\, \B(N)\big) = 1> 0$.

\subsection*{\cref{wasserstein-stability} for general $n$}
The difficulty of extending \cref{wasserstein-stability} to higher dimensions lies in the fact that, in our proof, we use key results of Bjerkevik and Lesnick that apply to the cases $n\in\{1,2\}$, and for which we currently do not know of any $n$-parameter generalizations.

\subsection*{Topologies of the spaces of signed barcodes and of Hilbert functions}
Let $n \in \N$, and let $\persmod$ denote the set of isomorphism classes of finitely presentable $n$-parameter persistence modules.
Although the dissimilarity $\widehat{d_B}$ does not satisfy the triangle inequality, it can be used to define a topology on $\sbarcodes$, the set of finite signed barcodes that are the Betti signed barcode of some finitely presentable persistence module.
For the basis of the topology one uses balls as done for defining the topology of a metric space.
With this definition, it follows from \cref{main-theorem} that the Betti signed barcode operator $\sbetti : \persmod \to \sbarcodes$ is continuous.

We can then topologize the set $\hils$ of all Hilbert functions of finitely presentable persistence modules, using the final topology induced by the map $\Sigma : \sbarcodes \to \hils$ given by 
$\Sigma(\B_+,\B_-) = \hil\left(\bigoplus_{i \in \B_+} F_i\right) - \hil\left(\bigoplus_{j \in \B_-} F_j\right)$.
We refer to this topology on $\hils$ as the \emph{bottleneck topology}.
Since $\hil = \Sigma \circ \sbetti$, it follows that $\hil : \persmod \to \hils$ is continuous.
The upshot is that $\sbarcodes$ is not $T_1$, by \cref{corollary:same-hilbert-function}, whereas $\hils$ endowed with the bottleneck topology is $T_1$, since, if $\hil(M) \neq \hil(N)$, then there is a ball (with respect to $\widehat{d_B}$) around $\sbetti(M)$ that does not contain $\sbetti(N)$.
In fact, the map $\Sigma : \sbarcodes \to \hils$ is the universal map onto a $T_1$ space.
Note, however, that the bottleneck topology is not Hausdorff, by \cref{possible-issue-example}.

Another topology can be defined on $\hils$ by an analogous construction, but using $\widehat{\dwass}$ instead of $\widehat{d_B}$.
We refer to this topology on $\hils$ as the \emph{$1$-Wasserstein topology}.
It is better behaved than the bottleneck topology.
Indeed, the $1$-Wasserstein topology is Hausdorff, and in fact it is metrizable, as it can be metrized using the extended distance $\dwasshils$ of \cref{distance-on-hilbert-functions}.
As we do in \cref{hilbert-stability-corollary} for the case of two-parameter persistence, one can use $\dwasshils$ to state a stability result for Hilbert functions.

\bibliographystyle{siamplain}
\bibliography{biblio}

\end{document}